\documentclass[10pt]{amsart}

\usepackage{amsmath,amsthm,amssymb}
\usepackage[mathscr]{eucal}
\usepackage{graphicx,array,pinlabel,enumerate,enumitem}
\usepackage[small,bf,margin=24pt]{caption}
\usepackage{graphicx}
\graphicspath{ {images/} }
\usepackage{tikz}
\usepackage{float}

\theoremstyle{plain}
\newtheorem{thm}{Theorem}[section]
\newtheorem{prop}[thm]{Proposition}
\newtheorem{lemma}[thm]{Lemma}
\newtheorem{cor}[thm]{Corollary}

\theoremstyle{definition}\newtheorem{Def}[thm]{Definition}
\theoremstyle{remark}\newtheorem{rmk}[thm]{Remark}
\newtheorem{ex}[thm]{Example}

\numberwithin{equation}{section}

\newcommand*{\Z}{\ensuremath{\mathbf Z}}

\begin{document}

\title
{Clock theorems for triangulated surfaces}

\author{Camden Hine}

\address{Georgia Institute of Technology}
\email{camdenhine@gatech.edu}

\author{Tam\'as K\'alm\'an}

\address{Tokyo Institute of Technology}
\email{kalman@math.titech.ac.jp}
\urladdr{www.math.titech.ac.jp/\char126 kalman}

\date{}

\keywords{}


\begin{abstract}
We investigate 
triangulations of the two-dimensional sphere and torus with the faces properly colored white and black.
We focus on matchings between white triangles and incident vertices. On the torus our objects are perfect pairings, whereas on the sphere this is only true after removing one triangle and its vertices. In the latter case, such matchings (first studied by Tutte) extend the notion of state in
Kauffman's formal knot theory and we show that his Clock Theorem, in its form due to Gilmer and Litherland, also extends: the set of matchings naturally forms a distributive lattice. 
Here the role of state transposition is played by a simple local operation about black triangles. 
By contrast, on the torus, the analogous state transition graph is usually 
disconnected: some of its components still form distributive lattices 
with global maxima and minima, 
while other components contain directed cycles and are without local extrema.
\end{abstract}

\maketitle

\section{Introduction}
\label{sec:intro}

In this paper we investigate properly three-colored triangulations of 
orientable surfaces, in particular of the sphere and the torus. These objects, also known as Latin bitrades, have been actively researched, see, e.g., \cite{cw,bm}. Our results have two sources of motivation. One is topological: such triangulations can be viewed as an extension of (isotopy classes of) generic immersed curves and in the spherical case, our main result generalizes Kauffman's Clock Theorem \cite{kauffman}, which is an important step in his proof of the state expansion of the Alexander polynomial. The other is combinatorial: again in the spherical case, Tutte defined certain bijections between objects within the triangulation and used them to establish his Tree Trinity Theorem \cite{tutte1,tutte3}. These bijections generalize Kauffman's states. After merging these two threads, it becomes natural to move our investigation from the sphere to the torus, where in a sense the situation is even more symmetric and beautiful.

We will call the colors of the vertices of the triangulation red, green, and blue. Triangles are colored white or black according to the cyclic order of the three colors around them. It is also natural to color each edge with the `third color' not occurring among its endpoints. We will refer to such a structure as a \emph{trinity}.

It is easy to see that in each trinity red edges, green edges, blue edges, white triangles, and black triangles all have the same number $n$. The number of vertices exceeds $n$ by the Euler characteristic of the surface. 

Tutte's Tree Trinity Theorem \cite{tutte1,tutte3} (see Theorem \ref{ttt} below) concerns trinities on $S^2$. He gave a remarkable proof using matchings between vertices and incident white triangles. To adjust the sizes of the two sets, he excluded an `outer' white triangle and its three vertices. 
In the case of the torus such a thing is not necessary: one may consider \emph{Tutte matchings} between all white triangles and all vertices (so that pairs are incident), both of which form sets of size $n$. These two cases (\emph{planar} and \emph{toric} trinities) are the focus of this paper.

Following Tutte, we will indicate matches between triangles and vertices by drawing an arrow across the triangle and terminating at the vertex. At the tail end, the arrows are extended beyond the white triangle and into the adjacent black one, all the way to the opposite vertex. That way the two ends of the arrow, and the edge that it crosses, are all of the same color. We let the arrow inherit this color as well. In other words, arrows of color $X$ are chosen from among the edges of the dual graph $G_X^*$, with respect to our surface, of the graph $G_X$ formed by the edges of the trinity of color $X$. Note that as $G_X$ is bipartite, $G_X^*$ is naturally oriented, exactly by the rule (the same for all three colors $X$) that the tail end of each arrow is in a black triangle.

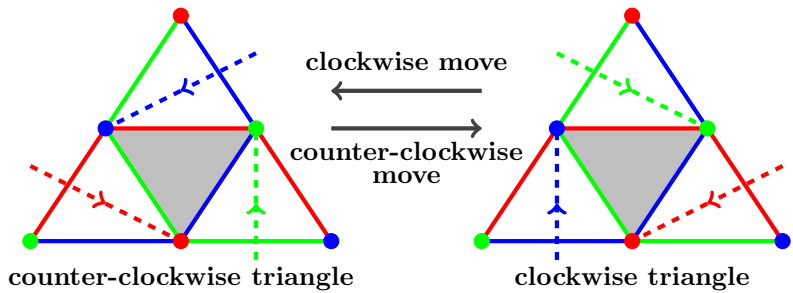
\begin{figure}[h]
\begin{tikzpicture}[scale=.5]
\path [fill=lightgray] (2,4) -- (6,4) -- (4,1);

\draw [ultra thick,red] (0,1) -- (2,4);
\draw [ultra thick,red] (2,4) -- (6,4);
\draw [ultra thick,red] (6,4) -- (8,1);
\draw [ultra thick,blue] (0,1) -- (4,1);
\draw [ultra thick,blue] (4,1) -- (6,4);
\draw [ultra thick,blue] (6,4) -- (4,7);
\draw [ultra thick,green] (4,7) -- (2,4);
\draw [ultra thick,green] (2,4) -- (4,1);
\draw [ultra thick,green] (4,1) -- (8,1);

\draw [->,ultra thick,blue,dashed] (6,6) to (4,5);
\draw [ultra thick,blue,dashed] (4,5) to (2,4);
\draw [->,ultra thick,green,dashed] (6,.5) to (6,2);
\draw [ultra thick,green,dashed] (6,2) to (6,4);
\draw [->,ultra thick,red,dashed] (0,3) to (2,2);
\draw [ultra thick,red,dashed] (2,2) to (4,1);

\draw [fill=green,green] (0,1) circle [radius=0.2];
\draw [fill=green,green] (6,4) circle [radius=0.2];
\draw [fill=red,red] (4,1) circle [radius=0.2];
\draw [fill=red,red] (4,7) circle [radius=0.2];
\draw [fill=blue,blue] (2,4) circle [radius=0.2];
\draw [fill=blue,blue] (8,1) circle [radius=0.2];

\node at (4,0) {\textbf{counter-clockwise triangle}};

\begin{scope}[shift={(12,0)}]
\path [fill=lightgray] (2,4) -- (6,4) -- (4,1);

\draw [ultra thick,red] (0,1) -- (2,4);
\draw [ultra thick,red] (2,4) -- (6,4);
\draw [ultra thick,red] (6,4) -- (8,1);
\draw [ultra thick,blue] (0,1) -- (4,1);
\draw [ultra thick,blue] (4,1) -- (6,4);
\draw [ultra thick,blue] (6,4) -- (4,7);
\draw [ultra thick,green] (4,7) -- (2,4);
\draw [ultra thick,green] (2,4) -- (4,1);
\draw [ultra thick,green] (4,1) -- (8,1);

\draw [->,ultra thick,blue,dashed] (2,.5) to (2,2);
\draw [ultra thick,blue,dashed] (2,2) to (2,4);
\draw [->,ultra thick,green,dashed] (2,6) to (4,5);
\draw [ultra thick,green,dashed] (4,5) to (6,4);
\draw [->,ultra thick,red,dashed] (8,3) to (6,2);
\draw [ultra thick,red,dashed] (6,2) to (4,1);

\draw [fill=green,green] (0,1) circle [radius=0.2];
\draw [fill=green,green] (6,4) circle [radius=0.2];
\draw [fill=red,red] (4,1) circle [radius=0.2];
\draw [fill=red,red] (4,7) circle [radius=0.2];
\draw [fill=blue,blue] (2,4) circle [radius=0.2];
\draw [fill=blue,blue] (8,1) circle [radius=0.2];

\node at (4,0) {\textbf{clockwise triangle}};
\end{scope}

\draw [->,ultra thick,darkgray] (12,5) to (8,5);
\draw [->,ultra thick,darkgray] (8,4) to (12,4);
\node at (10,3.4) {\textbf{counter-clockwise}};
\node at (10,2.7) {\textbf{move}};
\node at (10,5.8) {\textbf{clockwise move}};

\end{tikzpicture}
\caption{State transition about an empty black triangle.}
\label{fig:clockmove}
\end{figure}

The arrows that constitute a Tutte matching may not cross each other over white triangles but may cross over black ones. On the other hand, black triangles may also be \emph{empty} (void of arrows), something that does not happen to the white ones, except the outer triangle in the planar case.

We will also refer to Tutte matchings as \emph{states}, a term borrowed from statistical mechanics via knot theory. We will pay particular attention to the following \emph{state transitions}, inspired directly (as explained in subsection \ref{sec:kauffman}) by Kauffman's state transpositions (clock moves) of his formal knot theory \cite{kauffman}. When a state has an empty black triangle, its vertices must be matched with the three adjacent white triangles. There are exactly two ways to do this and we refer to the empty black triangle as \emph{clockwise} or \emph{counterclockwise} accordingly. Now, a \emph{clockwise move} turns a clockwise triangle into a counter-clockwise one, as seen in Figure \ref{fig:clockmove}. Its inverse is a \emph{counter-clockwise move}. (We will also allow other moves, see Figure \ref{fig:general_clock} and subsection \ref{sec:decomp}, but this is the basic idea.)

We define a \emph{state transition graph}, where the vertices correspond to the states (matchings), and two vertices are connected if it is possible to go from one to the other by performing a clockwise or counter-clockwise move. We also record the direction of our moves, wherein moving `up' in the state transition graph will correspond to a clockwise move, and moving `down' will correspond to a counter-clockwise move. A (local) maximum is a state that admits only counter-clockwise moves and (local) minima are defined analogously.

We will distinguish between two kinds of connected components of the state transition graph, cyclic and acyclic. 
A component is \emph{cyclic} if there exists a state in it to which it is possible to return after a non-trivial series of only clockwise moves. Let us call such a state \emph{recurrent}. Components without recurrent states will be called \emph{acyclic}.

We will see that acyclic components always have a global maximum and a global minimum; what is more, they carry the structure of a distributive lattice. Furthermore, the state transition graph of a planar trinity is always acyclic (which is easy to see) and connected (which is much less obvious). This extends theorems of Kauffman \cite{kauffman} and of Gilmer and Litherland \cite{gl}, which make the same claims for the particular kind of planar trinity that arises from Kauffman's `universes.' Let us note that in \cite{kauffman} the connectedness result is a crucial step in proving the state expansion formula for the Alexander polynomial --- a theorem that became even more significant when it was used to show that knot Floer homology categorifies the Alexander polynomial.

By contrast, cyclic components (which may only occur in toric cases) cannot have local maxima or minima. We will in fact show that in a cyclic component every state is recurrent after performing exactly one clockwise move on each black triangle. Cyclic components do not form distributive lattices for obvious reasons, but one may define infinite cyclic covers of them which do.

The paper is organized as follows. In subsection \ref{sec:defs} we set our definitions and in \ref{sec:kauffman} we explain the original inspiration for our work. In subsection \ref{sec:decomp} we deal with the possibility of pasting (planar) trinities into others and in \ref{sec:lemmas} we start proving our results by collecting some observations that apply equally to the sphere and torus. Section \ref{sec:acyclic} establishes the distributive lattice structure of acyclic components, again so that it applies to both genera. Then we separate cases and in Sections \ref{sec:sphere} and \ref{sec:torus} we show our main results for genus $0$ and $1$, respectively.

\emph{Acknowledgments}: The research reported in this paper was carried out while the first author visited the Tokyo Institute of Technology under the Young Scientist Exchange Program.
The second author is supported by a Japan Society for the Promotion of Science (JSPS) Grant-in-Aid for Scientific Research C (no.\ 17K05244).

\section{Preliminaries}
\label{sec:prelim}

\subsection{Trinities, states, and transitions}
\label{sec:defs}

Let $\Sigma$ be a compact, oriented, closed surface of genus $g$. Let us fix a triangulation $\mathcal T$ of $\Sigma$ in which all vertices are colored red, green, or blue. We will denote the sets of vertices of the various colors by $R$ (for red) $E$ (for green) and $V$ (for blue)\footnote{The latter two, perhaps unusual choices were first made in \cite{hypertutte} motivated by hypergraph considerations. To make them easier to memorize, in that paper the color names emerald and violet were adopted. (`Red' and `region' conveniently share the same first letter.) We choose not to do so here but will keep the symbols.}. We call $\mathcal T$ a \emph{trinity} if vertices of the same color are never incident to the same edge. In trinities, edges may also be three-colored: between green and blue vertices we have red edges forming the bipartite graph $G_R$ and there is a similarly defined green graph $G_E$ and blue graph $G_V$. All three graphs are cellularly embedded in $\Sigma$ so that each region of $G_R$ contains a unique point of $R$ and so on. Furthermore, each of the three embedded graphs determines $\mathcal T$ up to isotopy.

The orientation of $\Sigma$ induces a two-coloring of the triangles of $\mathcal T$. Namely, we decide which triangles are black and which are white by the rule that in the black triangles, green comes after blue when going around the edges (or the vertices) in a clockwise direction, and the opposite is true in the white triangles. 

Since all black (as well as all white) triangles have a unique edge of each color, it is easy to see that the same natural number $n=n(\mathcal T)$ gives the size of the set of black triangles, white triangles, red edges, green edges, and blue edges. By using the cellulation $G_R$ of $\Sigma$, we find that
\[\chi(\Sigma)=2-2g=|E|+|V|-n+|R|.\]
Thus if $\Sigma$ is the torus, for any $\mathcal T$, the number $|E|+|V|+|R|$ of vertices and the number $n$ of white triangles agree. 

For other surfaces, of which we will only consider the sphere, this is not so. When $\Sigma$ is the sphere, we will single out one white triangle, to be called \emph{outer}, and its vertices that we call the \emph{roots}, so that non-outer white triangles and non-root vertices are equinumerous. We will use stereographic projection from a point in the outer white triangle to produce planar drawings. In fact, we will refer to spherical trinities with a fixed outer triangle as \emph{planar}.

\begin{Def}
Let $\mathcal T$ be a toric or planar trinity. 
A \emph{Tutte matching} or \emph{state} of $\mathcal T$ is a perfect matching between white triangles and incident vertices. (In planar cases we should add `non-outer' and `non-root,' but we will not always do so explicitly.)
\end{Def}

We will represent matches between triangles and vertices by arrows, as explained in the Introduction.

\begin{figure}[h]
\begin{tikzpicture}[scale=.5]

\path [fill=lightgray] (2,4) -- (3.5, 2.75) -- (4,1);
\path [fill=lightgray] (4,1) -- (4.5, 2.75) -- (6,4);
\path [fill=lightgray] (6,4) -- (4,3.5) -- (2,4);

\draw [ultra thick,red] (0,1) -- (2,4);
\draw [ultra thick,red] (2,4) -- (6,4);
\draw [ultra thick,red] (6,4) -- (8,1);
\draw [ultra thick,blue] (0,1) -- (4,1);
\draw [ultra thick,blue] (4,1) -- (6,4);
\draw [ultra thick,blue] (6,4) -- (4,7);
\draw [ultra thick,green] (4,7) -- (2,4);
\draw [ultra thick,green] (2,4) -- (4,1);
\draw [ultra thick,green] (4,1) -- (8,1);

\draw [ultra thick,green] (4,1) -- (4.5, 2.75);
\draw [ultra thick,green] (2,4) -- (4,3.5);
\draw [ultra thick,blue] (4,1) -- (3.5, 2.75);
\draw [ultra thick,blue] (6,4) -- (4,3.5);
\draw [ultra thick,red] (6,4) -- (4.5, 2.75);
\draw [ultra thick,red] (2,4) -- (3.5, 2.75);

\draw [->,ultra thick,blue,dashed] (6,6) to (4,5);
\draw [ultra thick,blue,dashed] (4,5) to (2,4);
\draw [->,ultra thick,green,dashed] (6,0) to (6,2);
\draw [ultra thick,green,dashed] (6,2) to (6,4);
\draw [->,ultra thick,red,dashed] (0,3) to (2,2);
\draw [ultra thick,red,dashed] (2,2) to (4,1);

\draw [fill=green,green] (0,1) circle [radius=0.2];
\draw [fill=green,green] (6,4) circle [radius=0.2];
\draw [fill=red,red] (4,1) circle [radius=0.2];
\draw [fill=red,red] (4,7) circle [radius=0.2];
\draw [fill=blue,blue] (2,4) circle [radius=0.2];
\draw [fill=blue,blue] (8,1) circle [radius=0.2];

\draw [fill=green,green] (3.5, 2.75) circle [radius=0.2];
\draw [fill=blue,blue] (4.5, 2.75) circle [radius=0.2];
\draw [fill=red,red] (4,3.5) circle [radius=0.2];

\begin{scope}[shift={(12,0)}]

\path [fill=lightgray] (2,4) -- (3.5, 2.75) -- (4,1);
\path [fill=lightgray] (4,1) -- (4.5, 2.75) -- (6,4);
\path [fill=lightgray] (6,4) -- (4,3.5) -- (2,4);

\draw [ultra thick,red] (0,1) -- (2,4);
\draw [ultra thick,red] (2,4) -- (6,4);
\draw [ultra thick,red] (6,4) -- (8,1);
\draw [ultra thick,blue] (0,1) -- (4,1);
\draw [ultra thick,blue] (4,1) -- (6,4);
\draw [ultra thick,blue] (6,4) -- (4,7);
\draw [ultra thick,green] (4,7) -- (2,4);
\draw [ultra thick,green] (2,4) -- (4,1);
\draw [ultra thick,green] (4,1) -- (8,1);

\draw [ultra thick,green] (4,1) -- (4.5, 2.75);
\draw [ultra thick,green] (2,4) -- (4,3.5);
\draw [ultra thick,blue] (4,1) -- (3.5, 2.75);
\draw [ultra thick,blue] (6,4) -- (4,3.5);
\draw [ultra thick,red] (6,4) -- (4.5, 2.75);
\draw [ultra thick,red] (2,4) -- (3.5, 2.75);

\draw [->,ultra thick,blue,dashed] (2,0) to (2,2);
\draw [ultra thick,blue,dashed] (2,2) to (2,4);
\draw [->,ultra thick,green,dashed] (2,6) to (4,5);
\draw [ultra thick,green,dashed] (4,5) to (6,4);
\draw [->,ultra thick,red,dashed] (8,3) to (6,2);
\draw [ultra thick,red,dashed] (6,2) to (4,1);

\draw [fill=green,green] (0,1) circle [radius=0.2];
\draw [fill=green,green] (6,4) circle [radius=0.2];
\draw [fill=red,red] (4,1) circle [radius=0.2];
\draw [fill=red,red] (4,7) circle [radius=0.2];
\draw [fill=blue,blue] (2,4) circle [radius=0.2];
\draw [fill=blue,blue] (8,1) circle [radius=0.2];

\draw [fill=green,green] (3.5, 2.75) circle [radius=0.2];
\draw [fill=blue,blue] (4.5, 2.75) circle [radius=0.2];
\draw [fill=red,red] (4,3.5) circle [radius=0.2];

\end{scope}

\draw [->,ultra thick,darkgray] (12,5) to (8,5);
\draw [->,ultra thick,darkgray] (8,4) to (12,4);
\node at (10,3.4) {\textbf{counter-clockwise}};
\node at (10,2.7) {\textbf{move}};
\node at (10,5.8) {\textbf{clockwise move}};

\end{tikzpicture}
\caption{The general case of the clock move.}
\label{fig:general_clock}
\end{figure}
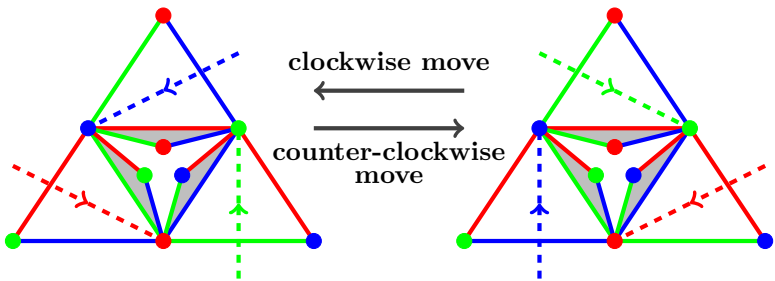

We shall now turn to the central idea of the paper, namely \emph{state transitions}. As we will explain shortly, this grows directly out of Kauffman's notion of a `state transposition.' Suppose that $\mathcal T$ contains three white triangles and three vertices arranged as shown in Figure \ref{fig:general_clock}, where we also assume that the three edges incident to our objects surround a disk (not containing the three white triangles) embedded in $\Sigma$. Now if in a given state the three triangles and three vertices are matched among themselves, then we may switch their assignments as shown in Figure \ref{fig:general_clock} and apply terminology (clockwise and counter-clockwise moves) as explained in the Introduction. 

\begin{rmk}
One may be tempted to consider longer similar chains of vertices and white triangles, of length some multiple of $3$. But a simple argument using Euler characteristic reveals that if these vertices and triangles are to be matched to each other, then for the corresponding cycle of edges to be separating on $\Sigma$, it is necessary that the cycle have length $3$ and even then the region not containing the white triangles can only be a disk. On the other hand, when $\Sigma$ has positive genus, non-separating chains of any length $3k$ are entirely possible. In 
the proof of Theorem \ref{thm:one_turn} 
we will see that this is essentially why acyclic components exist in toric cases. We save the investigation of clock moves about non-separating cycles for future work.
\end{rmk}

Of course, the simplest scenario regarding state transitions is when the three white triangles surround a single black one, which will then be called a (clockwise or counter-clockwise) \emph{empty black triangle} for reasons explained in the Introduction. In that case we will also refer to our operation as \emph{changing} 
the empty black triangle. When the region contains more than just one black triangle, we have a case of a nontrivial \emph{connected sum} --- see subsection \ref{sec:decomp} for more on those.

We define the \emph{state transition graph} of a trinity as in the Introduction.

\begin{ex}
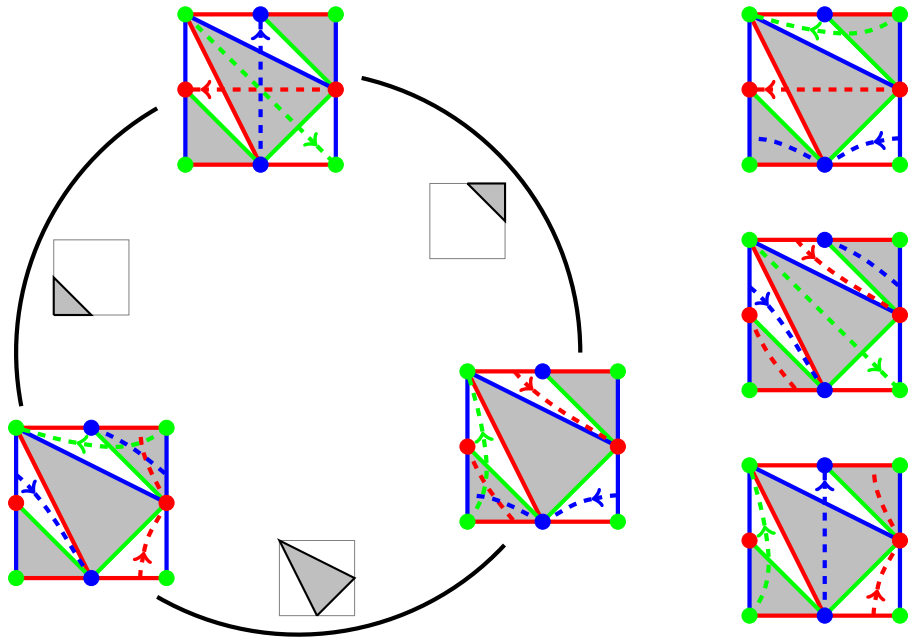
\begin{figure}[h]
\begin{tikzpicture} [scale=.25]

\draw [ultra thick] (15,0) arc [radius=15, start angle=0, end angle=77];
\draw [ultra thick] (-7.5,13) arc [radius=15, start angle=120, end angle=191];
\draw [ultra thick] (-7.5,-13) arc [radius=15, start angle=240, end angle=317];

\draw [help lines] (7,5) rectangle (11,9);
\draw [thick,fill=lightgray] (9,9) -- (11,9) -- (11,7) -- (9,9);
\draw [help lines] (-1,-14) rectangle (3,-10);
\draw [thick,fill=lightgray] (1,-14) -- (-1,-10) -- (3,-12) -- (1,-14);
\draw [help lines] (-13,2) rectangle (-9,6);
\draw [thick,fill=lightgray] (-13,2) -- (-11,2) -- (-13,4) -- (-13,2);

\begin{scope}[shift={(-6,10)}]

\path [fill=lightgray] (0,0) -- (4,0) -- (0,4);
\path [fill=lightgray] (4,0) -- (8,4) -- (0,8);
\path [fill=lightgray] (8,4) -- (8,8) -- (4,8);

\draw [ultra thick,red] (0,0) -- (8,0);
\draw [ultra thick,red] (0,8) -- (8,8);
\draw [ultra thick,red] (0,8) -- (4,0);
\draw [ultra thick,green] (0,4) -- (4,0) -- (8,4) -- (4,8);
\draw [ultra thick,blue] (0,0) -- (0,8);
\draw [ultra thick,blue] (8,0) -- (8,8);
\draw [ultra thick,blue] (0,8) -- (8,4);

\draw [fill=red,red] (0,4) circle [radius=0.4];
\draw [fill=red,red] (8,4) circle [radius=0.4];
\draw [fill=green,green] (0,0) circle [radius=0.4];
\draw [fill=green,green] (8,0) circle [radius=0.4];
\draw [fill=green,green] (0,8) circle [radius=0.4];
\draw [fill=green,green] (8,8) circle [radius=0.4];
\draw [fill=blue,blue] (4,0) circle [radius=0.4];
\draw [fill=blue,blue] (4,8) circle [radius=0.4];

\draw [->,ultra thick,red,dashed] (8,4) -- (.8,4);
\draw [ultra thick,red,dashed] (.8,4) -- (0,4);
\draw [->,ultra thick,green,dashed] (0,8) -- (7,1);
\draw [ultra thick,green,dashed] (7,1) -- (8,0);
\draw [->,ultra thick,blue,dashed] (4,0) -- (4,7.2);
\draw [ultra thick,blue,dashed] (4,7.2) -- (4,8);

\end{scope}

\begin{scope}[shift={(9,-9)}]

\path [fill=lightgray] (0,0) -- (4,0) -- (0,4);
\path [fill=lightgray] (4,0) -- (8,4) -- (0,8);
\path [fill=lightgray] (8,4) -- (8,8) -- (4,8);

\draw [ultra thick,red] (0,0) -- (8,0);
\draw [ultra thick,red] (0,8) -- (8,8);
\draw [ultra thick,red] (0,8) -- (4,0);
\draw [ultra thick,green] (0,4) -- (4,0) -- (8,4) -- (4,8);
\draw [ultra thick,blue] (0,0) -- (0,8);
\draw [ultra thick,blue] (8,0) -- (8,8);
\draw [ultra thick,blue] (0,8) -- (8,4);

\draw [fill=red,red] (0,4) circle [radius=0.4];
\draw [fill=red,red] (8,4) circle [radius=0.4];
\draw [fill=green,green] (0,0) circle [radius=0.4];
\draw [fill=green,green] (8,0) circle [radius=0.4];
\draw [fill=green,green] (0,8) circle [radius=0.4];
\draw [fill=green,green] (8,8) circle [radius=0.4];
\draw [fill=blue,blue] (4,0) circle [radius=0.4];
\draw [fill=blue,blue] (4,8) circle [radius=0.4];

\draw [ultra thick,red,dashed] (0,4) to [out=-70,in=130] (2.5,0);
\draw [->,ultra thick,red,dashed] (2.5,8) to [out=-50,in=140] (3.5,7);
\draw [ultra thick,red,dashed] (3.5,7) to [out=-40,in=145] (8,4);
\draw [->,ultra thick,green,dashed] (0,0) to [out=50,in=-80] (0.8,4.8);
\draw [ultra thick,green,dashed] (0.8,4.8) to [out=100,in=-75] (0,8);
\draw [ultra thick,blue,dashed] (4,0) to [out=155,in=0] (0,1.4);
\draw [->,ultra thick,blue,dashed] (8,1.4) to [out=180,in=10] (6.6,1.2);
\draw [ultra thick,blue,dashed] (6.6,1.2) to [out=190,in=25] (4,0);

\end{scope}

\begin{scope}[shift={(-15,-12)}]

\path [fill=lightgray] (0,0) -- (4,0) -- (0,4);
\path [fill=lightgray] (4,0) -- (8,4) -- (0,8);
\path [fill=lightgray] (8,4) -- (8,8) -- (4,8);

\draw [ultra thick,red] (0,0) -- (8,0);
\draw [ultra thick,red] (0,8) -- (8,8);
\draw [ultra thick,red] (0,8) -- (4,0);
\draw [ultra thick,green] (0,4) -- (4,0) -- (8,4) -- (4,8);
\draw [ultra thick,blue] (0,0) -- (0,8);
\draw [ultra thick,blue] (8,0) -- (8,8);
\draw [ultra thick,blue] (0,8) -- (8,4);

\draw [fill=red,red] (0,4) circle [radius=0.4];
\draw [fill=red,red] (8,4) circle [radius=0.4];
\draw [fill=green,green] (0,0) circle [radius=0.4];
\draw [fill=green,green] (8,0) circle [radius=0.4];
\draw [fill=green,green] (0,8) circle [radius=0.4];
\draw [fill=green,green] (8,8) circle [radius=0.4];
\draw [fill=blue,blue] (4,0) circle [radius=0.4];
\draw [fill=blue,blue] (4,8) circle [radius=0.4];

\draw [ultra thick,red,dashed] (8,4) to [out=115,in=-90] (6.6,8);
\draw [->,ultra thick,red,dashed] (6.6,0) to [out=90,in=-100] (6.8,1.4);
\draw [ultra thick,red,dashed] (6.8,1.4) to [out=80,in=-115] (8,4);
\draw [->,ultra thick,green,dashed] (8,8) to [out=220,in=-10] (3.2,7.2);
\draw [ultra thick,green,dashed] (3.2,7.2) to [out=170,in=-15] (0,8);
\draw [ultra thick,blue,dashed] (4,8) to [out=-20,in=140] (8,5.5);
\draw [->,ultra thick,blue,dashed] (0,5.5) to [out=-40,in=130] (1,4.5);
\draw [ultra thick,blue,dashed] (1,4.5) to [out=-50,in=125] (4,0);

\end{scope}

\begin{scope}[shift={(24,10)}]

\path [fill=lightgray] (0,0) -- (4,0) -- (0,4);
\path [fill=lightgray] (4,0) -- (8,4) -- (0,8);
\path [fill=lightgray] (8,4) -- (8,8) -- (4,8);

\draw [ultra thick,red] (0,0) -- (8,0);
\draw [ultra thick,red] (0,8) -- (8,8);
\draw [ultra thick,red] (0,8) -- (4,0);
\draw [ultra thick,green] (0,4) -- (4,0) -- (8,4) -- (4,8);
\draw [ultra thick,blue] (0,0) -- (0,8);
\draw [ultra thick,blue] (8,0) -- (8,8);
\draw [ultra thick,blue] (0,8) -- (8,4);

\draw [fill=red,red] (0,4) circle [radius=0.4];
\draw [fill=red,red] (8,4) circle [radius=0.4];
\draw [fill=green,green] (0,0) circle [radius=0.4];
\draw [fill=green,green] (8,0) circle [radius=0.4];
\draw [fill=green,green] (0,8) circle [radius=0.4];
\draw [fill=green,green] (8,8) circle [radius=0.4];
\draw [fill=blue,blue] (4,0) circle [radius=0.4];
\draw [fill=blue,blue] (4,8) circle [radius=0.4];

\draw [->,ultra thick,red,dashed] (8,4) -- (.8,4);
\draw [ultra thick,red,dashed] (.8,4) -- (0,4);
\draw [->,ultra thick,green,dashed] (8,8) to [out=220,in=-10] (3.2,7.2);
\draw [ultra thick,green,dashed] (3.2,7.2) to [out=170,in=-15] (0,8);
\draw [ultra thick,blue,dashed] (4,0) to [out=155,in=0] (0,1.4);
\draw [->,ultra thick,blue,dashed] (8,1.4) to [out=180,in=10] (6.6,1.2);
\draw [ultra thick,blue,dashed] (6.6,1.2) to [out=190,in=25] (4,0);

\end{scope}

\begin{scope}[shift={(24,-2)}]

\path [fill=lightgray] (0,0) -- (4,0) -- (0,4);
\path [fill=lightgray] (4,0) -- (8,4) -- (0,8);
\path [fill=lightgray] (8,4) -- (8,8) -- (4,8);

\draw [ultra thick,red] (0,0) -- (8,0);
\draw [ultra thick,red] (0,8) -- (8,8);
\draw [ultra thick,red] (0,8) -- (4,0);
\draw [ultra thick,green] (0,4) -- (4,0) -- (8,4) -- (4,8);
\draw [ultra thick,blue] (0,0) -- (0,8);
\draw [ultra thick,blue] (8,0) -- (8,8);
\draw [ultra thick,blue] (0,8) -- (8,4);

\draw [fill=red,red] (0,4) circle [radius=0.4];
\draw [fill=red,red] (8,4) circle [radius=0.4];
\draw [fill=green,green] (0,0) circle [radius=0.4];
\draw [fill=green,green] (8,0) circle [radius=0.4];
\draw [fill=green,green] (0,8) circle [radius=0.4];
\draw [fill=green,green] (8,8) circle [radius=0.4];
\draw [fill=blue,blue] (4,0) circle [radius=0.4];
\draw [fill=blue,blue] (4,8) circle [radius=0.4];

\draw [ultra thick,red,dashed] (0,4) to [out=-70,in=130] (2.5,0);
\draw [->,ultra thick,red,dashed] (2.5,8) to [out=-50,in=140] (3.5,7);
\draw [ultra thick,red,dashed] (3.5,7) to [out=-40,in=145] (8,4);
\draw [->,ultra thick,green,dashed] (0,8) -- (7,1);
\draw [ultra thick,green,dashed] (7,1) -- (8,0);
\draw [ultra thick,blue,dashed] (4,8) to [out=-20,in=140] (8,5.5);
\draw [->,ultra thick,blue,dashed] (0,5.5) to [out=-40,in=130] (1,4.5);
\draw [ultra thick,blue,dashed] (1,4.5) to [out=-50,in=125] (4,0);

\end{scope}

\begin{scope}[shift={(24,-14)}]

\path [fill=lightgray] (0,0) -- (4,0) -- (0,4);
\path [fill=lightgray] (4,0) -- (8,4) -- (0,8);
\path [fill=lightgray] (8,4) -- (8,8) -- (4,8);

\draw [ultra thick,red] (0,0) -- (8,0);
\draw [ultra thick,red] (0,8) -- (8,8);
\draw [ultra thick,red] (0,8) -- (4,0);
\draw [ultra thick,green] (0,4) -- (4,0) -- (8,4) -- (4,8);
\draw [ultra thick,blue] (0,0) -- (0,8);
\draw [ultra thick,blue] (8,0) -- (8,8);
\draw [ultra thick,blue] (0,8) -- (8,4);

\draw [fill=red,red] (0,4) circle [radius=0.4];
\draw [fill=red,red] (8,4) circle [radius=0.4];
\draw [fill=green,green] (0,0) circle [radius=0.4];
\draw [fill=green,green] (8,0) circle [radius=0.4];
\draw [fill=green,green] (0,8) circle [radius=0.4];
\draw [fill=green,green] (8,8) circle [radius=0.4];
\draw [fill=blue,blue] (4,0) circle [radius=0.4];
\draw [fill=blue,blue] (4,8) circle [radius=0.4];

\draw [ultra thick,red,dashed] (8,4) to [out=115,in=-90] (6.6,8);
\draw [->,ultra thick,red,dashed] (6.6,0) to [out=90,in=-100] (6.8,1.4);
\draw [ultra thick,red,dashed] (6.8,1.4) to [out=80,in=-115] (8,4);
\draw [->,ultra thick,green,dashed] (0,0) to [out=50,in=-80] (0.8,4.8);
\draw [ultra thick,green,dashed] (0.8,4.8) to [out=100,in=-75] (0,8);
\draw [->,ultra thick,blue,dashed] (4,0) -- (4,7.2);
\draw [ultra thick,blue,dashed] (4,7.2) -- (4,8);

\end{scope}

\end{tikzpicture}
\caption{The state transition graph of the smallest toric trinity.}
\label{fig:small}
\end{figure}

The smallest trinity on the torus, of $n=3$ vertices and white triangles, has six Tutte matchings. In Figure \ref{fig:small} we show the trinity using the standard identifications along the boundary of a square. Three of the six states have no empty black triangles, i.e., they are isolated points of the state transition graph. These are shown on the right of Figure \ref{fig:small}. The other three form a cycle. We labelled each edge of the cycle with the triangle that is being changed.
\end{ex}

\subsection{From Kauffman states to Tutte matchings}
\label{sec:kauffman}

Let us briefly explain the origin of our setup. None of our later arguments build on this subsection.

Let us consider a generic immersed curve $\gamma\subset\Sigma$ so that each of its complementary regions is a disk. (If $\Sigma$ is a sphere, this just amounts to saying that $\gamma$ is nonempty and connected.) Such curves give rise to a special class of trinities as shown in Figure \ref{fig:kauffman}. Namely, fix a checkerboard coloring of $\Sigma\setminus\gamma$ with the colors red and blue and place a vertex of like color in each region. Put green vertices to the double points of $\gamma$. Edges are defined in the obvious way: green vertices are connected to the surrounding region markers, and those in turn are joined by an edge for each segment of $\gamma$ incident to the two corresponding regions. Notice that in the resulting trinity each green vertex is incident to exactly four edges, something that we do not require in general trinities.

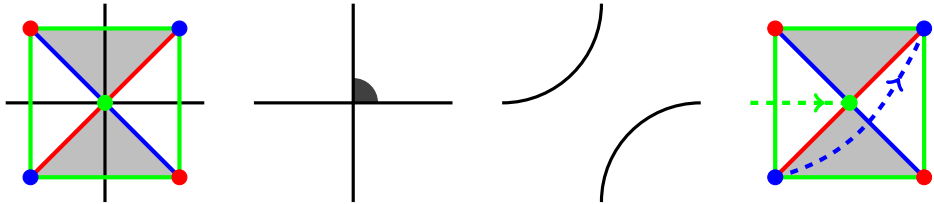
\begin{figure}[h]
\begin{tikzpicture}[scale=.33]
\path [fill=lightgray] (-3,3) -- (0,0) -- (3,3) -- (-3,3);
\path [fill=lightgray] (-3,-3) -- (0,0) -- (3,-3) -- (-3,-3);
\draw [very thick] (-4,0) -- (4,0);
\draw [very thick] (0,-4) -- (0,4);
\draw [ultra thick,red] (-3,-3) -- (3,3);
\draw [ultra thick,blue] (-3,3) -- (3,-3);
\draw [ultra thick,green] (-3,-3) -- (-3,3) -- (3,3) -- (3,-3) -- (-3,-3);
\draw [fill=green,green] (0,0) circle [radius=0.3];
\draw [fill=red,red] (-3,3) circle [radius=0.3];
\draw [fill=red,red] (3,-3) circle [radius=0.3];
\draw [fill=blue,blue] (-3,-3) circle [radius=0.3];
\draw [fill=blue,blue] (3,3) circle [radius=0.3];

\path [fill=darkgray] (10,1) -- (10,0) -- (11,0) to [out=90,in=0] (10,1);
\draw [very thick] (6,0) -- (14,0);
\draw [very thick] (10,-4) -- (10,4);

\draw [very thick] (24,0) arc [radius=4,start angle=90,end angle=180];
\draw [very thick] (16,0) arc [radius=4,start angle=-90,end angle=0];

\path [fill=lightgray] (27,3) -- (30,0) -- (33,3) -- (27,3);
\path [fill=lightgray] (27,-3) -- (30,0) -- (33,-3) -- (27,-3);
\draw [ultra thick,red] (27,-3) -- (33,3);
\draw [ultra thick,blue] (27,3) -- (33,-3);
\draw [ultra thick,green] (27,-3) -- (27,3) -- (33,3) -- (33,-3) -- (27,-3);
\draw [fill=green,green] (30,0) circle [radius=0.3];
\draw [fill=red,red] (27,3) circle [radius=0.3];
\draw [fill=red,red] (33,-3) circle [radius=0.3];
\draw [fill=blue,blue] (27,-3) circle [radius=0.3];
\draw [fill=blue,blue] (33,3) circle [radius=0.3];
\draw [->,ultra thick,blue,dashed] (27,-3) to [out=15,in=240] (32,1);
\draw [ultra thick,blue,dashed] (32,1) to [out=60,in=250] (33,3);
\draw [->,ultra thick,green,dashed] (26,0) -- (29,0);
\draw [ultra thick,green,dashed] (29,0) -- (30,0);
\end{tikzpicture}
\caption{From left to right: a crossing of $\gamma$ and the construction of the associated trinity; a Kauffman state marker; the local portion of the Euler-Jordan trail; the matching of the two nearby white triangles.}
\label{fig:kauffman}
\end{figure}

A \emph{Kauffman state} \cite{kauffman} is a matching between double points of $\gamma$ and incident regions. 
Each vertex-region pair is represented by marking, as in Figure \ref{fig:kauffman}, the quadrant at the vertex that the region contains. 
(To be more exact, Kauffman states are defined for spherical curves with two adjacent (`starred') regions removed from consideration.)

To each Kauffman state we associate a Tutte matching as follows. If a white triangle does not intersect any marked quadrants, we match it with its green vertex. Otherwise we match it with the red or blue region that contains the marked quadrant. (One of the white triangles straddling the starred regions (in most cases this is unique) is declared the outer triangle and is not matched.)

This is the `obvious choice' when we consider how to find triangles to match to red and blue vertices, but perhaps our construction is most natural in light of the Euler-Jordan trails that Kauffman associates to his states. See Figure \ref{fig:kauffman}. The trail results from splitting each crossing along its marker and the idea is that such a trail separates two spanning trees, in the red and blue Tait graphs\footnote{In these graphs red (resp.\ blue) region markers are joined by edges representing adjacency at crossings. They are dual to each other. In knot theory, they are commonly called the black and white graphs of a knot diagram.},
respectively. Indeed, when we represent the matchings involving red and blue vertices with arrows, as explained in the Introduction, then those arrows connect two vertices that are \emph{not} separated by the Euler-Jordan trail. Hence the unions of our red and blue arrows will be exactly the two spanning trees. Moreover, as the blue and the red starred regions are not matched, they have no arrow pointing to them; other regions have exactly one such arrow. That is, in our version the two spanning trees appear as (spanning) \emph{arborescences} rooted at the starred regions. See Section \ref{sec:sphere} for how this idea develops further.

Now the key point is that if we perform a \emph{state transposition} as in Figure \ref{fig:classic_clock} (top), the arrow configuration about the central black triangle (see bottom) changes as in Figure \ref{fig:clockmove}. In the case when the segment of $\gamma$ connecting the two crossings involved in the transposition intersects other parts of $\gamma$ (cf.\ \cite[Figure 5]{kauffman}), the change in the associated Tutte matchings will be of the more general type depicted in Figure \ref{fig:general_clock}. Notice that a clockwise motion of the Kauffman state markers results in a counter-clockwise turn of the three arrows (about their respective heads) involved in the corresponding clock move. (This makes it challenging to choose terminology that is both internally consistent and consistent with usage in \cite{kauffman,gl,ct}.)

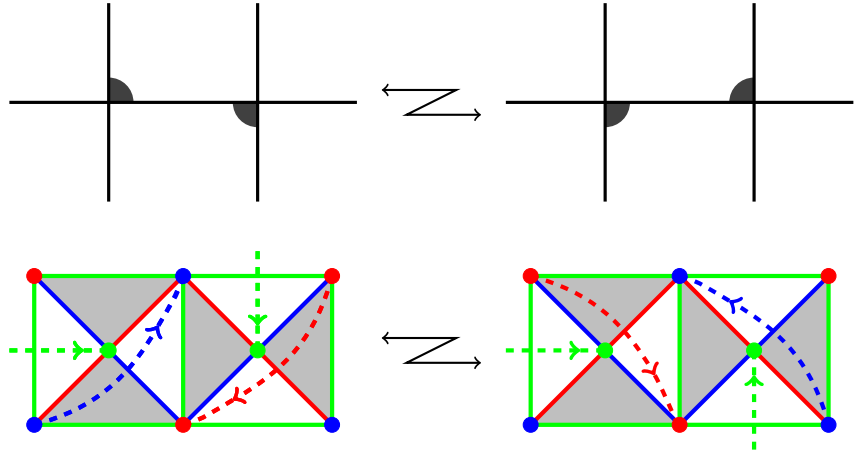
\begin{figure}[h]
\begin{tikzpicture}[scale=.33]
\path [fill=darkgray] (4,11) -- (4,10) -- (5,10) to [out=90,in=0] (4,11);
\path [fill=darkgray] (9,10) -- (10,10) -- (10,9) to [out=180,in=270] (9,10);
\draw [very thick] (0,10) -- (14,10);
\draw [very thick] (4,6) -- (4,14);
\draw [very thick] (10,6) -- (10,14);

\draw [<->,thick] (15,10.5) -- (18,10.5) -- (16,9.5) -- (19,9.5);

\path [fill=darkgray] (24,9) -- (24,10) -- (25,10) to [out=270,in=0] (24,9);
\path [fill=darkgray] (29,10) -- (30,10) -- (30,11) to [out=180,in=90] (29,10);
\draw [very thick] (20,10) -- (34,10);
\draw [very thick] (24,6) -- (24,14);
\draw [very thick] (30,6) -- (30,14);

\path [fill=lightgray] (1,3) -- (4,0) -- (7,3) -- (1,3);
\path [fill=lightgray] (1,-3) -- (4,0) -- (7,-3) -- (1,-3);
\path [fill=lightgray] (7,3) -- (10,0) -- (7,-3) -- (7,3);
\path [fill=lightgray] (13,3) -- (10,0) -- (13,-3) -- (13,3);
\draw [ultra thick,red] (1,-3) -- (7,3) -- (13,-3);
\draw [ultra thick,blue] (1,3) -- (7,-3) -- (13,3);
\draw [ultra thick,green] (1,-3) -- (1,3) -- (7,3) -- (7,-3) -- (1,-3);
\draw [ultra thick,green] (7,3) -- (13,3) -- (13,-3) -- (7,-3);
\draw [->,ultra thick,blue,dashed] (1,-3) to [out=15,in=240] (6,1);
\draw [ultra thick,blue,dashed] (6,1) to [out=60,in=250] (7,3);
\draw [->,ultra thick,red,dashed] (13,3) to [out=255,in=30] (9,-2);
\draw [ultra thick,red,dashed] (9,-2) to [out=210,in=15] (7,-3);
\draw [->,ultra thick,green,dashed] (0,0) -- (3,0);
\draw [ultra thick,green,dashed] (3,0) -- (4,0);
\draw [->,ultra thick,green,dashed] (10,4) -- (10,1);
\draw [ultra thick,green,dashed] (10,1) -- (10,0);
\draw [->,ultra thick,blue,dashed] (1,-3) to [out=15,in=240] (6,1);
\draw [ultra thick,blue,dashed] (6,1) to [out=60,in=255] (7,3);
\draw [->,ultra thick,red,dashed] (13,3) to [out=255,in=30] (9,-2);
\draw [ultra thick,red,dashed] (9,-2) to [out=210,in=15] (7,-3);
\draw [->,ultra thick,green,dashed] (0,0) -- (3,0);
\draw [ultra thick,green,dashed] (3,0) -- (4,0);
\draw [->,ultra thick,green,dashed] (10,4) -- (10,1);
\draw [ultra thick,green,dashed] (10,1) -- (10,0);
\draw [fill=green,green] (4,0) circle [radius=0.3];
\draw [fill=green,green] (10,0) circle [radius=0.3];
\draw [fill=red,red] (1,3) circle [radius=0.3];
\draw [fill=red,red] (7,-3) circle [radius=0.3];
\draw [fill=red,red] (13,3) circle [radius=0.3];
\draw [fill=blue,blue] (1,-3) circle [radius=0.3];
\draw [fill=blue,blue] (7,3) circle [radius=0.3];
\draw [fill=blue,blue] (13,-3) circle [radius=0.3];

\draw [<->,thick] (15,.5) -- (18,.5) -- (16,-.5) -- (19,-.5);

\path [fill=lightgray] (21,3) -- (24,0) -- (27,3) -- (21,3);
\path [fill=lightgray] (21,-3) -- (24,0) -- (27,-3) -- (21,-3);
\path [fill=lightgray] (27,3) -- (30,0) -- (27,-3) -- (27,3);
\path [fill=lightgray] (33,3) -- (30,0) -- (33,-3) -- (33,3);
\draw [ultra thick,red] (21,-3) -- (27,3) -- (33,-3);
\draw [ultra thick,blue] (21,3) -- (27,-3) -- (33,3);
\draw [ultra thick,green] (21,-3) -- (21,3) -- (27,3) -- (27,-3) -- (21,-3);
\draw [ultra thick,green] (27,3) -- (33,3) -- (33,-3) -- (27,-3);
\draw [->,ultra thick,red,dashed] (21,3) to [out=-15,in=120] (26,-1);
\draw [ultra thick,red,dashed] (26,-1) to [out=300,in=285] (27,-3);
\draw [->,ultra thick,blue,dashed] (33,-3) to [out=105,in=-30] (29,2);
\draw [ultra thick,blue,dashed] (29,2) to [out=150,in=-15] (27,3);
\draw [->,ultra thick,green,dashed] (20,0) -- (23,0);
\draw [ultra thick,green,dashed] (23,0) -- (24,0);
\draw [->,ultra thick,green,dashed] (30,-4) -- (30,-1);
\draw [ultra thick,green,dashed] (30,-1) -- (30,0);
\draw [fill=green,green] (24,0) circle [radius=0.3];
\draw [fill=green,green] (30,0) circle [radius=0.3];
\draw [fill=red,red] (21,3) circle [radius=0.3];
\draw [fill=red,red] (27,-3) circle [radius=0.3];
\draw [fill=red,red] (33,3) circle [radius=0.3];
\draw [fill=blue,blue] (21,-3) circle [radius=0.3];
\draw [fill=blue,blue] (27,3) circle [radius=0.3];
\draw [fill=blue,blue] (33,-3) circle [radius=0.3];
\end{tikzpicture}
\caption{A transposition of Kauffman states and the corresponding clock move of Tutte matchings.}
\label{fig:classic_clock}
\end{figure}

Thus planar trinities and their clock moves generalize Kauffman's universes and his state transpositions. It is most useful to think of this generalization as a replacement of the two Tait \emph{graphs}
with a pair of planar dual \emph{hypergraphs} \cite[Section 8]{hypertutte}: The hyperedges are the green points and indeed, in general they can have more than two blue (and hence red) neighbors. The point of view taken by Cohen and Teicher in \cite{ct} can be seen as a significant step toward this generalization.

One neat consequence of our reinterpretation of Kauffman's ideas is that it applies equally well, or maybe even better, to the torus. See Section \ref{sec:torus}. It remains to be seen if the study of states and state transitions leads to new insights when applied to knot projections on the torus, such as grid diagrams.

\subsection{Decomposing trinities}
\label{sec:decomp}

Next, we explore the following basic idea. Given a spherical trinity $\mathcal T_0$ with an outer white triangle and another trinity $\mathcal T$ on any surface, we may choose an arbitrary black triangle of $\mathcal T$ and replace it with $\mathcal T_0$ minus its outer triangle to form the new trinity $\mathcal T\#\mathcal T_0$. (There is a unique way to match colors along the common triangular boundary.) Let us call this the \emph{connected sum} of $\mathcal T$ and $\mathcal T_0$. 
Now if $\mathcal T$ is on the sphere (plane) or on the torus, so that we can talk about state transition graphs, we observe the following.

\begin{thm}
\label{thm:graph_is_multiplicative}
The state transition graph of $\mathcal T\#\mathcal T_0$ is the product of the state transition graphs of $\mathcal T$ and $\mathcal T_0$, respectively. That is, the set of states for $\mathcal T\#\mathcal T_0$ can be identified with the set of pairs $(s,s')$, where $s$ is a state for $\mathcal T$ and $s'$ is a state for $\mathcal T_0$; furthermore, there is a clockwise move from $(s,s')$ to $(t,t')$ if and only if there is a clockwise move from $s$ to $t$ in $\mathcal T$ or there is one from $s'$ to $t'$ in $\mathcal T_0$. 
\end{thm}

\begin{proof}
Consider the trinities $\mathcal T$ and $\mathcal T_0$ which form the connected sum $\mathcal T\#\mathcal T_0$ by gluing $\mathcal T_0$ into the black triangle $\Delta$ of $\mathcal T$, as described at the beginning of this subsection. 
Notice that even after gluing, the non-root vertices of $\mathcal T_0$ are only adjacent to non-outer white triangles of $\mathcal T_0$ and not to any white triangles of $\mathcal T$. Since the non-root vertices and the non-outer white triangles of $\mathcal T_0$ are equinumerous, we have that in any state of the connected sum $\mathcal T\#\mathcal T_0$, those two sets are matched to each other, forming a state $s'$ of $\mathcal T_0$. All other vertices and white triangles of the connected sum come from $\mathcal T$ and are matched to each other, resulting in a state $s$ of $\mathcal T$. The state transitions of the connected sum correspond to state transitions of $\mathcal T$ (including moves about $\Delta$) or $\mathcal T_0$. This is because if both a non-root vertex of $\mathcal T_0$ and a vertex of $\mathcal T$ were involved in a move, then one of the roots (call it $x$) would have to be too, furthermore one of the three white triangles of the move would have to be from $\mathcal T_0$ and adjacent to $x$. But as we have just seen, such triangles can not be matched to $x$ neither before nor after the move. So all three vertices are from `the same side' and hence so are the three triangles, again because they have to be matched to those vertices before the move.
\end{proof}

We say that the connected sum $\mathcal T\#\mathcal T_0$ is \emph{trivial} if it involves, either as $\mathcal T$ or as $\mathcal T_0$, the spherical trinity $\mathcal F$ with just one black and one white triangle. We call a trinity \emph{irreducible} if it is different from $\mathcal F$ and can not be written as a nontrivial connected sum. Connected sums can be iterated by gluing furher planar trinities into other black triangles, including black triangles that belong to previously glued-in parts. For any trinity $\mathcal T\ne\mathcal F$ on any surface, there is a well-defined irreducible core so that $\mathcal T$ is obtained by gluing a `tree' of planar trinities to the core. This is because the triangles, along which our trinities are glued, can only form trivial intersections.

With Theorem \ref{thm:graph_is_multiplicative} in hand, for most of the rest of the paper we will assume that our trinity is irreducible. That assumption implies that all clock moves take place about empty black triangles as in Figure \ref{fig:clockmove} and not as in Figure \ref{fig:general_clock}. For connected sums, all we have to say follows more or less immediately from Theorem \ref{thm:graph_is_multiplicative} and our other claims.

\subsection{Basic observations}
\label{sec:lemmas}

We now start building the theory of general (irreducible) trinities, states, and clock moves. Let us a fix a trinity $\mathcal T$ on the surface $\Sigma$ (of genus $0$ or $1$).


In any state of $\mathcal T$ and for any vertex, there can be at most two empty black triangles adjacent to the vertex. (Namely the ones on either side of the white triangle that is matched to the vertex.) When there are two such triangles, they are oppositely oriented. In particular, 

\begin{lemma}
\label{lem:triangles_are_far}
Clockwise (counter-clockwise) empty black triangles may not share vertices.
\end{lemma}

A clockwise move on the triangle $\Delta$ can eliminate a(n empty) counter-clockwise triangle, but only if it shares a vertex with $\Delta$. By the previous paragraph, there can be at most three such counter-clockwise triangles.

\begin{lemma}
\label{lem:moves_are_independent}
For any pair of clockwise (counter-clockwise) empty black triangles in a given state, changing one to counter-clockwise (clockwise) leaves the other one still empty. Furthermore, the same state results (after two steps) regardless of the order in which we change the two triangles.
\end{lemma}

\begin{proof}
Both claims are immediate from Lemma \ref{lem:triangles_are_far} after noting that if a white triangle shared edges with two black triangles, then those black triangles would have a common vertex as well.
\end{proof}


By the same token (that is, a trivial induction proof), we also have the following.

\begin{lemma}
For any set of clockwise (counter-clockwise) empty black triangles in a state, changing all of them to counter-clockwise (clockwise) will result in the same state regardless of the order in which they are changed.
\end{lemma}

\begin{lemma}
\label{lem:swapping}
Suppose we have a path in the state transition graph, in which at some point  a clockwise (counter-clockwise) move is performed immediately after a counter-clockwise (clockwise) move. Then we can swap the order of these two moves without changing any other states or moves along the path. In particular, the starting and ending states remain the same.
\end{lemma}

\begin{proof}
Trivial from Lemma \ref{lem:moves_are_independent}: from the point of view of the state between the two moves to be swapped, both moves are clockwise (counter-clockwise). Thus a new state can be generated by performing them both as in Lemma \ref{lem:moves_are_independent}. The modified path will pass through this new state.
\end{proof}

Note that the operation of the lemma may lead to one or two canceling pairs of moves in the sequence and thus may shorten the path by two or four steps. The difference of the numbers of clockwise and counterclockwise moves however remains constant. The next lemma is another obvious consequence.

\begin{lemma}
\label{lem:path_construction}
Suppose two states are in the same connected component of the state transition graph. Then it is possible to connect them with a sequence of moves in which a series of counter-clockwise moves is followed by a series of clockwise ones (that is, a clockwise move is never followed by a counter-clockwise one).
\end{lemma}

Regarding cancelations of inverse moves, let us make the following observation.

\begin{lemma}
\label{lem:cancel}
If a path in the state transition graph contains a move and its inverse, then there exists another path between the same endpoints, made up of the same moves as before (possibly in a different sequence) minus the canceling pair.
\end{lemma}

\begin{proof}
Assume without loss of generality that the first occurrence of the move, $\Delta$, is clockwise and the second, $\Delta^{-1}$, is counter-clockwise. Let us use Lemma \ref{lem:swapping} to re-order the intervening steps of the path so that just one sequence of counter-clockwise moves is followed by a sequence of clockwise ones. (We may keep applying the Lemma as long as there exist counter-clockwise moves following clockwise ones in the subsequence.) Let us then further apply the Lemma to swap the order of $\Delta$ and the counter-clockwise moves following it, as well as to swap the order of $\Delta^{-1}$ and the clockwise moves preceding it. After this $\Delta$ and $\Delta^{-1}$ appear in adjacent positions and can be canceled.
\end{proof}

\section{Acyclic components}
\label{sec:acyclic}

In any trinity, we call a state \emph{recurrent} if there is a non-empty sequence of clockwise moves that starts and ends with it. A component of the state transition graph that contains a recurrent state is called \emph{cyclic} and other components are called \emph{acyclic}. In this section we focus on the structure of acyclic components and prove that they become distributive lattices in a natural way. In order to apply to both planar and toric cases, the argument (which follows Gilmer and Litherland \cite{gl} with a few added details) will be rather abstract.

Since we have only finitely many states, performing infinitely many counter-clockwise moves would result in a cycle. By reversing such a cycle we would find its states to be recurrent. Hence in an acyclic component it is not possible to do counter-clockwise moves indefinitely and thus local minima (states with only clockwise moves) exist.

Furthermore, 
such a local minimum is unique (and hence a global minimum). Indeed, if two existed then so would a path between them as in Lemma \ref{lem:path_construction}, which is an obvious contradiction.

Having found the unique minimal state $a$ in our acyclic component, our strategy will be to use it as a base point and to identify each state in the component, in an appropriate sense, with the set of moves that leads to it from $a$. This identification requires two separate arguments, Lemmas \ref{lem:set_of_moves-->state} and \ref{lem:state-->set_of_moves}.

\begin{lemma}
\label{lem:set_of_moves-->state}
When changing a given starting state through a series of (clockwise and counter-clockwise) clock moves, the same ending state will result as long as these same moves (with the same algebraic multiplicities) are performed in any allowable order.
\end{lemma}

\begin{proof}
Let us concentrate on an arbitrary vertex during our series of moves. The white triangle that it is matched to changes by a counter-clockwise turn for every clockwise move involving the vertex, and likewise by a clockwise turn for every counter-clockwise move. So the last triangle is the same for every order in which the moves are performed. (In fact, up to canceling pairs of moves, the order in which the triangles incident to our vertex are changed is the same in every allowable sequence of the given moves.)
\end{proof}

In acyclic components the converse holds as well.

\begin{lemma}
In an acyclic component of the state transition graph, if two paths connect the same endpoints, then the algebraic multiplicity (exponent sum) of each move is the same along the two paths.
\label{lem:state-->set_of_moves}
\end{lemma}

\begin{proof}
Let the two paths be $\Pi_1$ and $\Pi_2$ and consider the cycle $\Pi_1\Pi_2^{-1}$. We may modify this using Lemma \ref{lem:cancel} so that each black triangle that is changed along the cycle is always changed the same way. This does not affect exponent sums. Our goal is to show that such a cycle $\Omega$ has to be trivial. 

The black triangles of $\mathcal T$ are classified into three groups with respect to $\Omega$: those that only occur as clockwise moves, those that only occur as counter-clockwise moves, and those that are not changed. If neither of the first two groups is empty, then by the connectedness of the surface, there has to be a pair of black triangles from two different groups that share a vertex. The arrow pointing to this vertex will thus not be stationary through our sequence of moves, furthermore it is easy to see that it is impossible for it to return to its starting position. (Under our assumption moves of the arrow cannot be reversed, and once it reaches the `wrong kind of black triangle,' the arrow will get stuck, unable to complete a full turn.)

This contradiction ensures that $\Omega$ consists only of clockwise moves or only of counter-clockwise moves. Hence if $\Omega$ is non-trivial, then all states along it are recurrent. This completes the proof.
\end{proof}

If $a$ is the unique minimal state in our acyclic component of the state transition graph, Lemma \ref{lem:path_construction} guarantees that all states $s$ in the component can be reached from $a$ by a sequence $\Pi_s$ of clockwise moves. Lemmas \ref{lem:set_of_moves-->state} and \ref{lem:state-->set_of_moves} imply that $s$ may be identified with the function $\varphi_s$ that associates to a black triangle the number of times it is changed in $\Pi_s$. Note that the values of $\varphi_s$ need not all be $0$'s and $1$'s, cf.\ \cite[Remark on p.\ 243]{gl} and \cite{ct} for a detailed study in the plane curve case. 

Given any ground set, integer-valued functions defined on it form a distributive lattice under the operations of minimum and maximum. For us the ground set will be the set of black triangles of $\mathcal T$. For our acyclic component to inherit the distributive lattice structure, it suffices to prove the following Proposition.

\begin{prop}
For any acyclic component of the state transition graph, the image of the association $s\mapsto \varphi_s$, from the set of states in the component to integer-valued functions defined on black triangles, is closed under minima and maxima.
\end{prop}

\begin{proof}
Consider two states $s$ and $t$ of the acyclic component and their functions of multiplicities $\varphi_s$ and $\varphi_t$. Let $r$ be a state so that 
\begin{itemize}
\item both $s$ and $t$ can be reached from $r$ by clockwise moves, and
\item among such states, $\varphi_r$ has a maximal sum of values.
\end{itemize}
We claim that $\varphi_r=\min\{\,\varphi_s,\varphi_t\,\}$. The inequality $\varphi_r\le\min\{\,\varphi_s,\varphi_t\,\}$ is obvious. If it was strict, then there would be a clockwise move $\Delta$ that is necessary (that is, has to occur with positive multiplicity) to get from $r$ to $s$, as well as to get from $r$ to $t$. Let us fix four sequences of clockwise moves, $K$, $L$, $M$, and $N$, so that the word $K\Delta L$ describes a path from $r$ to $s$ and $M\Delta N$ is a path from $r$ to $t$. Without loss of generality we may assume that the moves comprising $M$ do not occur in $K$ or $L$, as well as that $\Delta$ does not occur in $K$ or $M$.

Now the following words, generated by repeated applications of Lemma \ref{lem:swapping}, all describe paths from $s$ to $t$:
\begin{multline*}
L^{-1}\Delta^{-1}K^{-1}M\Delta N 
\to L^{-1}M\Delta^{-1}K^{-1}\Delta N 
\to L^{-1}M\Delta^{-1}\Delta K^{-1}N \\
\to L^{-1}MK^{-1}N \to L^{-1}K^{-1}MN 
\to L^{-1}K^{-1}M\Delta^{-1}\Delta N
\to L^{-1}K^{-1}\Delta^{-1}M\Delta N.
\end{multline*}
The end of the last word, $M\Delta N$, tells us that the path described by the word passes through the state $r$. 
The beginning of the last word, $L^{-1}K^{-1}\Delta^{-1}$, shows that the move $\Delta$ may be applied to $r$ and the resulting state is such that both $s$ (using the word $KL$) and $t$ (using the word $MN$, as guaranteed by the word $L^{-1}K^{-1}MN$ in the above sequence) can be reached from it by clockwise moves. This contradicts the definition of $r$ and completes our proof regarding minima.

Maxima can be handled by `turning everything upside down': An acyclic component has a unique maximum state $b$ too, and all states $s$ can be uniquely described by the multiplicities of moves, encoded in the function $\varphi_b-\varphi_s$, required to reach $b$ from them. By the same proof as above, 
\[\min\{\,\varphi_b-\varphi_s,\varphi_b-\varphi_t\,\}=\varphi_b-\max\{\,\varphi_s,\varphi_t\,\}\] 
describes a state for any $s$ and $t$.
\end{proof}

\begin{cor}
\label{cor:lattice}
Any acyclic component of the state transition graph has a distributive lattice structure in which the wedge and the join of the states $s$ and $t$ are the unique states described by the functions $\min\{\,\varphi_s,\varphi_t\,\}$ and $\max\{\,\varphi_s,\varphi_t\,\}$, respectively.
\end{cor}

\section{Planar trinities}
\label{sec:sphere}

In this section we will assume that our trinity $\mathcal T$ is drawn on the sphere and it is irreducible, i.e., it has $n(\mathcal T)\ge2$ and it can not be obtained as a nontrivial connected sum, as described in subsection \ref{sec:decomp}. (For the trivial trinity $\mathcal F$ of $n(\mathcal F)=1$, with its unique Tutte matching between two empty sets, all of our conclusions trivially hold.) We also fix an outer white triangle and call its vertices the roots.
We will 
represent trinities by planar drawings so that the outer white triangle becomes the unbounded region.

\begin{prop}
\label{prop:nocycle}
Planar trinities do not have recurrent states, i.e., all components of their state transition graphs are acyclic.
\end{prop}

\begin{proof}
Since the roots are not matched, it is not possible to perform any moves on the black triangles incident to them. Hence the arrows pointing to the non-root vertices that are incident to these triangles, 
if they move at all, will not be able to return to their starting positions by only clockwise moves.
So if there is a sequence of clockwise moves returning to the same state, then none of the black triangles incident to the above vertices can be involved in those moves. But then the same is true for all black triangles incident to a vertex of the previous ones, and so on. By the connectedness of $\mathcal T$, it follows that the sequence has to be trivial.
\end{proof}

The goal of this section is to show the following extension of a result of Kauffman \cite{kauffman}. The second claim, generalizing a theorem of Gilmer and Litherland \cite{gl}, follows readily from Proposition \ref{prop:nocycle} and Corollary \ref{cor:lattice}.

\begin{thm}
\label{thm:plane}
The state transition graph of a planar trinity has only one connected component. It has a distributive lattice structure as described in Corollary \ref{cor:lattice}.
\end{thm}

But before this let us recall a fundamental fact about planar trinities, Tutte's Tree Trinity Theorem \cite{tutte1}. A \emph{spanning arborescence} in a rooted directed graph is a spanning tree so that all its edges point away from the root. In particular, no edges point to the root and other vertices have a unique edge pointing to them, cf.\ \cite[Lemma 9.8]{hypertutte}.

\begin{thm}[Tutte]
\label{ttt}
The directed graphs $G_R^*$, $G_E^*$ and $G_V^*$ of a planar trinity $\mathcal T$ have the same number $\rho(\mathcal T)$ of spanning arborescences.
\end{thm}

This is in fact true with any choice of roots in the graphs. If we use the vertices of the outer white triangle as roots, a beautiful argument can be made based on Tutte matchings \cite{tutte3}. We summarize it in the following Lemma.

\begin{lemma}[Tutte]
\label{tuttelemma}
In a planar trinity $\mathcal T$, a matching between all non-root vertices of a given color and incident white triangles has at most one extension to a Tutte matching. An extension exists if and only if the arrows representing the partial matching form a spanning arborescence. 
In particular, the number of Tutte matchings in $\mathcal T$ is the \emph{magic number} $\rho(\mathcal T)$ of Theorem \ref{ttt}.
\end{lemma}

Let us recall just one step in the proof. A spanning arborescence $A\subset G_R^*$ is in particular a spanning tree, and hence so is its dual subgraph $\Gamma\subset G_R$. The latter always contains the edge $\kappa$ between the blue and green roots and starting from any vertex $x$ of $G_R$ there is a unique path $\Pi_x\subset\Gamma$ that leads to $\kappa$. (To be exact, let us say that $\Pi_x$ ends at an endpoint of $\kappa$ but it does not contain $\kappa$.) In the extension of $A$ to a Tutte matching, the white triangle matched to a blue or green vertex $x$ is the one adjacent to the first edge of $\Pi_x$.

By Lemma \ref{tuttelemma}, states and (triples of) spanning arborescences are synonymous in planar trinities. 
Starting from any state, one may perform clockwise moves until one runs out of options. A key observation 
is that the state at this point is uniquely determined by the trinity, given by a concrete construction. The construction appeared in \cite{hitoshi} where, following Kauffman, it was called the \emph{clocked arborescence}. More precisely, all three of $G_R^*$, $G_E^*$, and $G_V^*$ contain unique clocked arborescences and it turns out (Corollary \ref{cor:clockedstate}) that they fit together in the same \emph{clocked state} in the sense of Lemma \ref{tuttelemma}. It then immediately follows that the state transition graph of any planar trinity is connected, with a unique maximum (the clocked state) and, by symmetry, a unique minimum. 

The clocked arborescence is built by a greedy algorithm where one uses counter-clockwise turns at vertices (starting at the root and at the green side of the outer white triangle\footnote{The blue side is more practical, but this is consistent with what we will do later.}) to search for arrows (necessarily pointing away from the vertex) to be consecutively added to the arborescence. When a suitable arrow is found, the search continues at its terminal point. If a full turn has been performed at a vertex, the algorithm backtracks to the preceding vertex. 
This is equivalent to saying that after selecting each arrow, we return to the root and re-start the search, tracing counter-clockwise the perimeter of the subgraph selected thus far, until we meet the tail end of the first arrow that can be added without creating a cycle. 
See Figure \ref{fig:transition} for an example and \cite[Section 2]{hitoshi} for full details.

Let us point out the subtlety that the construction of the clocked arborescence operates at the tail ends of arrows, whereas clock moves are rotations of arrows about their heads. 

The idea of being built up from the root by adding on arrows one by one applies not just to the clocked one but to all (spanning) arborescences, at least as long as the rooted directed graph has a so-called ribbon structure. (Our directed graphs inherit ribbon structures from the sphere.) More precisely, for a (spanning) arborescence $A$, an ordering of its edges can be described as follows. We fix a sufficiently small neighborhood $N$ of $A$. Starting from its point along the green edge of the outer white triangle, we trace the boundary $\partial N$ in the counter-clockwise direction. Let us denote this closed path by $\Omega_A$. We number the edges of $A$ in the order of the first time $\Omega_A$ travels along (parallel to) them. Let us denote this ordering of the edges of $A$ by $<_A$. 

It is useful to extend $<_A$ to non-edges of $A$, too. (This idea is taken from Bernardi \cite{bernardi}.) If $A$ is indeed spanning, then as we traverse $\Omega_A$, we meet both ends of each non-edge of $A$ exactly once. We repeat our numbering above but this time we also include in our timeline the events when $\Omega_A$ crosses over 
the first of the two ends of a non-edge. 

In terms of this extended ordering, the clocked arborescence (among all spanning arborescences $A$) is characterized by the property that for each non-root vertex $y$ (of the same color as $A$), the edge of $A$ pointing to $y$ is smallest, with respect to $<_A$, in the set of the edges of the directed graph that terminate at $y$. (To put it simply, the clocked arborescence is indeed greedy: it grabs each vertex the first time it sees it. Other arborescences hesitate to do so in the case of at least one vertex.)

For the proof of our next proposition it will be important to pay attention to when, during the course of the construction of the arborescence $A\subset G_R^*$, the path $\Pi_x\subset G_R$ from $x$ to $\kappa$ (and hence its first edge and hence the white triangle matched to $x$) get decided. Here $x$ is a blue or a green vertex.

For this let us consider the process from yet another angle. At the beginning we have the entire graph $G_R^*$ (primal) on the one hand, and isolated green and blue vertices (dual) on the other. Next we remove the edge $\kappa^*$ (pointing to the red root across the outer white triangle) from $G_R^*$ and add $\kappa$ to the dual, thereby connecting the green and blue roots. (Notice that $\kappa^*$ is always the smallest edge in the extended ordering $<_A$.) Then we continue our walk along $\Omega_A$ and every time we cross over the first of the two ends of an edge $\varepsilon\in G_R^*\setminus A$, we remove it from the primal subgraph and add the dual edge $\varepsilon^*$ to the dual one. A trivial induction proof shows that this dual subgraph, at every stage starting from the second, has only one connected component that is not an isolated point. Since at the end of the process the primal subgraph is $A$ and the dual one is the spanning tree $\Gamma$, it follows that the dual subgraph is also cycle-free, that is a tree (plus isolated points), at every stage.

With this we see that if $\Omega_A$ crosses $\varepsilon$ for the first time near its tail, then a blue leaf gets added on to the dual subgraph; if it is near the head, then a green leaf. Indeed, the region of $G_R^*$ that $\Omega_A$ has just entered contains a blue or green vertex and this vertex must have been isolated before because otherwise we would have just created a cycle in the dual subgraph. We also see that (except for the initial edge $\varepsilon=\kappa^*$) the path $\Pi_x$ for the newly added leaf $x$ starts with $\varepsilon^*$; this is because such a path from $x$ to $\kappa$ exists in the current dual subgraph and hence in $\Gamma$, too. In particular, we have shown the following.

\begin{lemma}
\label{lem:firstedge}
For every edge $\varepsilon\in G_R^*\setminus A\setminus\{\kappa^*\}$, which has a green vertex $e$ to its right and a blue vertex $v$ to its left, if $\Omega_A$ crosses the tail end of $\varepsilon$ first, then $\varepsilon^*$ is the first edge along the path $\Pi_v$. If $\Omega_A$ crosses $\varepsilon$ near its head first, then $\varepsilon^*$ becomes the first edge along the path $\Pi_e$.
\end{lemma}

The most important technical result of this section is the following.

\begin{prop}
\label{prop:uniquemax}
If a state is such that one of its constituent spanning arborescences is not the clocked one for the respective directed graph, then it admits a clockwise move.
\end{prop}

\begin{cor}
\label{cor:clockedstate}
For any planar trinity, the union of the clocked arborescences of the graphs $G_R^*$, $G_E^*$, and $G_V^*$, all three viewed as a matching between white triangles and vertices of the respective color, is a Tutte matching.
\end{cor}

\begin{proof}
As arborescences exist (for example the clocked one), states exist by Lemma \ref{tuttelemma}. Then from Proposition \ref{prop:nocycle} we know that a state without clockwise moves exists. Finally by Proposition \ref{prop:uniquemax}, such a state has to consist of the three clocked arborescences. 
\end{proof}

The state of the three clocked arborescences 
then is the unique maximum of our state transition graph and thus finishes the proof of Theorem \ref{thm:plane}.

\begin{proof}[Proof of Proposition \ref{prop:uniquemax}]
Let us consider a state $s$ of our planar trinity $\mathcal T$ in which the arrows pointing to red vertices form a spanning arborescence $A$ of the directed graph $G_R^*$ that is different from the clocked arborescence. This means that edges $\varepsilon$ exist so that $\varepsilon\not\in A$ and the edge $\rho$ of $A$ that shares terminal points with $\varepsilon$ satisfies $\varepsilon<_A\rho$. Let us choose the largest one, with respect to $<_A$, among such edges and denote it by $\varepsilon$. Let $\rho$ be the edge of $A$ with the same terminal point $r$ as $\varepsilon$. See Figure \ref{fig:move_exists}. Since $\varepsilon<_A\rho$, when $\Omega_A$ first meets $\varepsilon$, it has not yet traveled along $\rho$ and thus $r$ has not been reached yet; therefore $\Omega_A$ will first reach the tail end of $\varepsilon$. (Informally, $\varepsilon$ is the last edge that gets `skipped over' when we build up $A$; after that $A$ is constructed in the greedy way, including the time of the selection of $\rho$.)

\begin{figure} [h]
\begin{tikzpicture}[scale=.33] 

\path [fill=lightgray] (0,0) to [out=120,in=-60] (-6,2) -- (-6,6) to [out=90,in=-90] (1,8) -- (5,5) -- (0,0);

\draw [red,ultra thick] (0,0) -- (10,0);
\draw [red,ultra thick,dotted] (0,0) to [out=120,in=-60] (-6,2);
\draw [red,ultra thick] (-6,2) -- (-6,6);
\draw [red,ultra thick,dotted] (-6,6) to [out=90,in=-90] (1,8);
\draw [red] (1,8) -- (5,10);
\draw [green,ultra thick] (0,0) -- (5,5);
\draw [blue,ultra thick] (1,8) -- (5,5);

\draw [->,red,dashed] (5,-5) -- (5,2);
\draw [red,dashed] (5,2) -- (5,5);
\draw [->,red,ultra thick,dashed] (1,13) -- (4,7);
\draw [red,ultra thick,dashed] (4,7) -- (5,5);
\draw [->,red,dashed] (-9,4) -- (-5,4);
\draw [red,dashed] (-5,4) -- (-3,4);

\draw [fill=red,red] (5, -5) circle [radius=0.3];
\draw [fill=red,red] (5, 5) circle [radius=0.3];
\draw [fill=red,red] (1, 13) circle [radius=0.3];
\draw [fill=red,red] (-9, 4) circle [radius=0.3];
\draw [fill=red,red] (-3, 4) circle [radius=0.3];
\draw [fill=green,green] (10, 0) circle [radius=0.3];
\draw [fill=green,green] (-6, 2) circle [radius=0.3];
\draw [fill=green,green] (1, 8) circle [radius=0.3];
\draw [fill=blue,blue] (0, 0) circle [radius=0.3];
\draw [fill=blue,blue] (-6, 6) circle [radius=0.3];
\draw [fill=blue,blue] (5, 10) circle [radius=0.3];

\node [right] at (5.3,5) {$r$};
\node [below] at (0,-.3) {$v$};
\node [below left] at (-6.1,1.9) {$e'$};
\node [left] at (-6.3,6) {$v'$};
\node [above left] at (.9,8.1) {$e$};
\node [right] at (-2.7,4) {$r'$};
\node [left] at (.7,13) {$q$};
\node [right] at (5,-3) {$\varepsilon$};
\node [above] at (8,0) {$\varepsilon^*$};
\node [above right] at (2,11) {$\rho$};
\node [below] at (-8,4) {$\varepsilon'$};
\node at (1,5) {$D$};

\end{tikzpicture}
    \caption{Finding a clockwise empty black triangle.}
    \label{fig:move_exists}
\end{figure}

Let $v$ be the blue vertex of $\mathcal T$ to the left of $\varepsilon$ and let $e$ be the green vertex to the right of $\rho$. These are both adjacent to $r$. As $A\subset G_R^*$ is a spanning tree, so is its dual subgraph $\Gamma\subset G_R$. In particular, $\Gamma$ contains a unique path $\Pi$ between $v$ and $e$. Let $\Pi$ and the edges $er$ and $rv$ form the cycle $\Theta$. As $\Pi\subset\Gamma$, edges of $A$ may not cross it; $A$ may only cross $\Theta$ at $r$.

The path $\Pi$ does not contain $\varepsilon^*$ because then $\Theta$ would separate the initial point $q$ of $\rho$ from the red root, making it necessary for a directed path to exist in $A$ from $r$ to $q$, which with $\rho$ would make a cycle in $A$.
Hence $\Theta$ surrounds a region $D$ of the sphere which does not contain $\varepsilon$ and $\rho$ (and hence it does not contain the red root and the outer white triangle either, so it appears in our projection as a bounded planar region).

We claim that $\Pi$ has to be a single edge. If it is not then there exists an edge along it, connecting the green vertex $e'$ and the blue vertex $v'$, so that these follow each other in the order $v,e',v',e$ along $\Pi$. The edge $e'v'$ is incident to a unique white triangle which is contained in $D$ and whose third (red) vertex will be denoted with $r'$. (It is possible that $r'=r$.) Let $\varepsilon'$ be the edge of $G_R^*$ that is dual to $e'v'$. 

By Lemma \ref{lem:firstedge}, the first edge of the path $\Pi_v$ is $\varepsilon^*$. We also see that $\Pi_{e'}$ and $\Pi_{v'}$ (as well as $\Pi_e$) follow $\Pi$ to $v$ and then travel along $\Pi_v$. To be consistent with Lemma \ref{lem:firstedge}, this is only possible if $\Omega_A$ first crosses $\varepsilon'$ at its tail end. It also implies that $\varepsilon<_A\varepsilon'$: indeed if $\varepsilon'$ came first in the order, then $v$ would still be an isolated point at the time of the crossing of $\varepsilon'$ by $\Omega_A$, making it impossible for $\Pi_{v'}$ to pass through $v$. To arrive at the desired contradiction (with the assumption that $\Pi$ is longer than one edge), we will argue that at the time of the first crossing of $\varepsilon'$ by $\Omega_A$, the point $r'$ has not yet been visited. That is, $\varepsilon'$ is also a `skipped edge,' which is not in line with the way we chose $\varepsilon$.

Indeed the first time that $r'$ is reached by an edge of $A$ is not before $r$ is reached by $\rho$ (recall that $A$ may enter $D$ only through $r$). At the time when $\Omega_A$ travels along $\rho$, the vertex $e$ cannot be isolated any more in the dual subgraph of the discussion before Lemma \ref{lem:firstedge}, hence $\Pi_e$ is already part of the dual subgraph, but that implies that $v'$ is not isolated either, meaning that the first crossing of $\varepsilon'$ by $\Omega_A$ happened even earlier.

Finally, we claim that the state $s$ admits a clockwise move about the disk $D$. This is clear from the presence of $\rho$ in $A$, Tutte's description of the extension of $A$ to a state (given after Lemma \ref{tuttelemma}), and the facts that $\Pi_v$ starts with $\varepsilon^*$ and that $\Pi_e$ starts with the unique edge of $\Pi$.
\end{proof}

\begin{ex}
\input{transition.tex}

Figure \ref{fig:transition} shows all states of a planar trinity and their state transition graph (distributive lattice). Clockwise empty black triangles are marked with a full dot and counter-clockwise ones are marked with a hollow dot. 
\end{ex}

\begin{rmk}
For a planar trinity $\mathcal T$, the magic number $\rho(\mathcal T)$ 
gives the cardinality of at least three 
structures: 
the distributive lattice of Theorem \ref{thm:plane}, the five isomorphic abelian groups studied in \cite{bm} (three of which are the sandpile groups 
of $G_R^*$, $G_E^*$, and $G_V^*$), and the set of bases for any of the six hypergraphical polymatroids induced by $\mathcal T$ \cite{hypertutte}. The interplay between these structures is not well understood.
\end{rmk}

\section{Toric trinities}
\label{sec:torus}

In this section we will assume that the trinity triangulates the torus. For the most part we also assume that it is irreducible in the sense of subsection \ref{sec:decomp}, although this is only strictly necessary for Proposition \ref{prop:same_turns}, Corollary \ref{reachability}, and Theorem \ref{thm:one_turn}. 

Recall that in Section \ref{sec:acyclic} we distinguished between cyclic and acyclic components of the state transition graph. Isolated or `fixed' points (states without allowed moves) are considered acyclic components. In contrast to the planar case, state transition graphs of toric trinities need not be connected. In fact, it is possible for both cyclic and acyclic components 
to arise from a single trinity, as seen in the following example.

\begin{ex}
\begin{figure} [h]
\begin{tikzpicture}[scale=.66] 
\path [fill=lightgray] (0,0) -- (2,0) -- (2,2);
\path [fill=lightgray] (0,4) -- (0,2) -- (2,2);
\path [fill=lightgray] (2,2) -- (2,4) -- (4,2);
\path [fill=lightgray] (2,2) -- (4,0) -- (6,0);
\path [fill=lightgray] (2,4) -- (4,4) -- (6,2);
\path [fill=lightgray] (6,0) -- (6,2) -- (4,2);
\path [fill=lightgray] (6,4) -- (6,2) -- (8,4);
\path [fill=lightgray] (6,2) -- (8,2) -- (8,0);
\draw [ultra thick, red] (0,2) -- (2,2);
\draw [ultra thick, red] (6,2) -- (8,2);
\draw [ultra thick, red] (2,4) -- (2,0);
\draw [ultra thick, red] (6,4) -- (6,0);
\draw [ultra thick, red] (2,4) -- (6,2);
\draw [ultra thick, red] (2,2) -- (6,0);
\draw [ultra thick, green] (0,4) -- (2,2);
\draw [ultra thick, green] (0,0) -- (2,2);
\draw [ultra thick, green] (8,4) -- (6,2);
\draw [ultra thick, green] (8,0) -- (6,2);
\draw [ultra thick, green] (4,4) -- (6,2);
\draw [ultra thick, green] (2,2) -- (4,0);
\draw [ultra thick, green] (6,2) -- (2,2);
\draw [ultra thick, blue] (0,0) -- (0,4) -- (8,4) -- (8,0) -- cycle;
\draw [ultra thick, blue] (2,4) -- (4,2);
\draw [ultra thick, blue] (4,2) -- (6,0);
\draw [->,ultra thick,red,dashed] (0,4) to [out=-60,in=80] (0.6,1.4);
\draw [ultra thick,red,dashed] (0.6,1.4) to [out=260,in=60] (0,0);
\draw [->,ultra thick,red,dashed] (0,0) to [out=30,in=170] (2.6,0.6); 
\draw [ultra thick,red,dashed] (2.6,0.6) to [out=-10,in=150] (4,0);
\draw [->,ultra thick,red,dashed] (4,4) -- (4,2.5);
\draw [ultra thick,red,dashed] (4,2.5) -- (4,2);
\draw [ultra thick,blue,dashed] (2,2) to [out=-35,in=130] (4.7,0);
\draw [->,ultra thick,blue,dashed] (4.7,4) to [out=-50,in=125] (5.3,3.3); 
\draw [ultra thick,blue,dashed] (5.3,3.3) to [out=-55,in=115] (6,2);
\draw [->,ultra thick,blue,dashed] (6,2) to [out=220,in=-10] (3.6,1.6);
\draw [ultra thick,blue,dashed] (3.6,1.6) to [out=170,in=-15] (2,2);
\draw [->,ultra thick,green,dashed] (0,2) -- (1.4,3.4);
\draw [ultra thick,green,dashed] (1.4,3.4) -- (2,4);
\draw [->,ultra thick,green,dashed] (6,4) -- (7.4,2.6); 
\draw [ultra thick,green,dashed] (7.4,2.6) -- (8,2);
\draw [->,ultra thick,green,dashed] (8,2) -- (6.6,0.6);
\draw [ultra thick,green,dashed] (6.6,0.6) -- (6,0);
\draw [fill=red,red] (0, 0) circle [radius=0.15];
\draw [fill=red,red] (0, 4) circle [radius=0.15];
\draw [fill=red,red] (4, 0) circle [radius=0.15];
\draw [fill=red,red] (4, 2) circle [radius=0.15];
\draw [fill=red,red] (4, 4) circle [radius=0.15];
\draw [fill=red,red] (8, 0) circle [radius=0.15];
\draw [fill=red,red] (8, 4) circle [radius=0.15];
\draw [fill=green,green] (0, 2) circle [radius=0.15];
\draw [fill=green,green] (2, 0) circle [radius=0.15];
\draw [fill=green,green] (2, 4) circle [radius=0.15];
\draw [fill=green,green] (6, 0) circle [radius=0.15];
\draw [fill=green,green] (6, 4) circle [radius=0.15];
\draw [fill=green,green] (8, 2) circle [radius=0.15];
\draw [fill=blue,blue] (2, 2) circle [radius=0.15];
\draw [fill=blue,blue] (6, 2) circle [radius=0.15];

\path [fill=lightgray] (10,0) -- (12,0) -- (12,2);
\path [fill=lightgray] (10,4) -- (10,2) -- (12,2);
\path [fill=lightgray] (12,2) -- (12,4) -- (14,2);
\path [fill=lightgray] (12,2) -- (14,0) -- (16,0);
\path [fill=lightgray] (12,4) -- (14,4) -- (16,2);
\path [fill=lightgray] (16,0) -- (16,2) -- (14,2);
\path [fill=lightgray] (16,4) -- (16,2) -- (18,4);
\path [fill=lightgray] (16,2) -- (18,2) -- (18,0);
\draw [ultra thick, red] (10,2) -- (12,2);
\draw [ultra thick, red] (16,2) -- (18,2);
\draw [ultra thick, red] (12,4) -- (12,0);
\draw [ultra thick, red] (16,4) -- (16,0);
\draw [ultra thick, red] (12,4) -- (16,2);
\draw [ultra thick, red] (12,2) -- (16,0);
\draw [ultra thick, green] (10,4) -- (12,2);
\draw [ultra thick, green] (10,0) -- (12,2);
\draw [ultra thick, green] (18,4) -- (16,2);
\draw [ultra thick, green] (18,0) -- (16,2);
\draw [ultra thick, green] (14,4) -- (16,2);
\draw [ultra thick, green] (12,2) -- (14,0);
\draw [ultra thick, green] (16,2) -- (12,2);
\draw [ultra thick, blue] (10,0) -- (10,4) -- (18,4) -- (18,0) -- cycle;
\draw [ultra thick, blue] (12,4) -- (14,2);
\draw [ultra thick, blue] (14,2) -- (16,0);
\draw [ultra thick,blue,dashed] (16,2) to [out=-25,in=180] (18,1.2);
\draw [->,ultra thick,blue,dashed] (10,1.2) to [out=0,in=190] (10.7,1.3);
\draw [ultra thick,blue,dashed] (10.7,1.3) to [out=10,in=205] (12,2);
\draw [->,ultra thick,blue,dashed] (12,2) to [out=40,in=170] (14.4,2.4);
\draw [ultra thick,blue,dashed] (14.4,2.4) to [out=-10,in=165] (16,2);
\draw [->,ultra thick,green,dashed] (10,2) -- (11.4,3.4);
\draw [ultra thick,green,dashed] (11.4,3.4) -- (12,4);
\draw [->,ultra thick,green,dashed] (16,4) -- (17.4,2.6); 
\draw [ultra thick,green,dashed] (17.4,2.6) -- (18,2);
\draw [->,ultra thick,green,dashed] (12,4) to [out=-15,in=205] (15.3,3.5);
\draw [ultra thick,green,dashed] (15.3,3.5) to [out=25,in=220] (16,4);
\draw [->,ultra thick,red,dashed] (10,0) to [out=30,in=170] (12.6,0.6); 
\draw [ultra thick,red,dashed] (12.6,0.6) to [out=-10,in=150] (14,0);
\draw [->,ultra thick,red,dashed] (14,0) -- (14,1.5);
\draw [ultra thick,red,dashed] (14,1.5) -- (14,2);
\draw [->,ultra thick,red,dashed] (14,2) -- (16.6,0.7);
\draw [ultra thick,red,dashed] (16.6,0.7) -- (18,0);
\draw [fill=red,red] (10, 0) circle [radius=0.15];
\draw [fill=red,red] (10, 4) circle [radius=0.15];
\draw [fill=red,red] (14, 0) circle [radius=0.15];
\draw [fill=red,red] (14, 2) circle [radius=0.15];
\draw [fill=red,red] (14, 4) circle [radius=0.15];
\draw [fill=red,red] (18, 0) circle [radius=0.15];
\draw [fill=red,red] (18, 4) circle [radius=0.15];
\draw [fill=green,green] (10, 2) circle [radius=0.15];
\draw [fill=green,green] (12, 0) circle [radius=0.15];
\draw [fill=green,green] (12, 4) circle [radius=0.15];
\draw [fill=green,green] (16, 0) circle [radius=0.15];
\draw [fill=green,green] (16, 4) circle [radius=0.15];
\draw [fill=green,green] (18, 2) circle [radius=0.15];
\draw [fill=blue,blue] (12, 2) circle [radius=0.15];
\draw [fill=blue,blue] (16, 2) circle [radius=0.15];

\path [fill=lightgray] (0,6) -- (2,6) -- (2,8);
\path [fill=lightgray] (0,10) -- (0,8) -- (2,8);
\path [fill=lightgray] (2,8) -- (2,10) -- (4,8);
\path [fill=lightgray] (2,8) -- (4,6) -- (6,6);
\path [fill=lightgray] (2,10) -- (4,10) -- (6,8);
\path [fill=lightgray] (6,6) -- (6,8) -- (4,8);
\path [fill=lightgray] (6,10) -- (6,8) -- (8,10);
\path [fill=lightgray] (6,8) -- (8,8) -- (8,6);
\draw [ultra thick, red] (0,8) -- (2,8);
\draw [ultra thick, red] (6,8) -- (8,8);
\draw [ultra thick, red] (2,10) -- (2,6);
\draw [ultra thick, red] (6,10) -- (6,6);
\draw [ultra thick, red] (2,10) -- (6,8);
\draw [ultra thick, red] (2,8) -- (6,6);
\draw [ultra thick, green] (0,10) -- (2,8);
\draw [ultra thick, green] (0,6) -- (2,8);
\draw [ultra thick, green] (8,10) -- (6,8);
\draw [ultra thick, green] (8,6) -- (6,8);
\draw [ultra thick, green] (4,10) -- (6,8);
\draw [ultra thick, green] (2,8) -- (4,6);
\draw [ultra thick, green] (6,8) -- (2,8);
\draw [ultra thick, blue] (0,6) -- (0,10) -- (8,10) -- (8,6) -- cycle;
\draw [ultra thick, blue] (2,10) -- (4,8);
\draw [ultra thick, blue] (4,8) -- (6,6);
\draw [fill=red,red] (0, 6) circle [radius=0.15];
\draw [fill=red,red] (0, 10) circle [radius=0.15];
\draw [fill=red,red] (4, 6) circle [radius=0.15];
\draw [fill=red,red] (4, 8) circle [radius=0.15];
\draw [fill=red,red] (4, 10) circle [radius=0.15];
\draw [fill=red,red] (8, 6) circle [radius=0.15];
\draw [fill=red,red] (8, 10) circle [radius=0.15];
\draw [fill=green,green] (0, 8) circle [radius=0.15];
\draw [fill=green,green] (2, 6) circle [radius=0.15];
\draw [fill=green,green] (2, 10) circle [radius=0.15];
\draw [fill=green,green] (6, 6) circle [radius=0.15];
\draw [fill=green,green] (6, 10) circle [radius=0.15];
\draw [fill=green,green] (8, 8) circle [radius=0.15];
\draw [fill=blue,blue] (2, 8) circle [radius=0.15];
\draw [fill=blue,blue] (6, 8) circle [radius=0.15];

\path [fill=lightgray] (10,6) -- (12,6) -- (12,8);
\path [fill=lightgray] (10,10) -- (10,8) -- (12,8);
\path [fill=lightgray] (12,8) -- (12,10) -- (14,8);
\path [fill=lightgray] (12,8) -- (14,6) -- (16,6);
\path [fill=lightgray] (12,10) -- (14,10) -- (16,8);
\path [fill=lightgray] (16,6) -- (16,8) -- (14,8);
\path [fill=lightgray] (16,10) -- (16,8) -- (18,10);
\path [fill=lightgray] (16,8) -- (18,8) -- (18,6);
\draw [ultra thick, red] (10,8) -- (12,8);
\draw [ultra thick, red] (16,8) -- (18,8);
\draw [ultra thick, red] (12,10) -- (12,6);
\draw [ultra thick, red] (16,10) -- (16,6);
\draw [ultra thick, red] (12,10) -- (16,8);
\draw [ultra thick, red] (12,8) -- (16,6);
\draw [ultra thick, green] (10,10) -- (12,8);
\draw [ultra thick, green] (10,6) -- (12,8);
\draw [ultra thick, green] (18,10) -- (16,8);
\draw [ultra thick, green] (18,6) -- (16,8);
\draw [ultra thick, green] (14,10) -- (16,8);
\draw [ultra thick, green] (12,8) -- (14,6);
\draw [ultra thick, green] (16,8) -- (12,8);
\draw [ultra thick, blue] (10,6) -- (10,10) -- (18,10) -- (18,6) -- cycle;
\draw [ultra thick, blue] (12,10) -- (14,8);
\draw [ultra thick, blue] (14,8) -- (16,6);
\draw [->,ultra thick,red,dashed] (10,10) to [out=-60,in=80] (10.6,7.4);
\draw [ultra thick,red,dashed] (10.6,7.4) to [out=260,in=60] (10,6);
\draw [->,ultra thick,red,dashed] (14,10) -- (14,8.5);
\draw [ultra thick,red,dashed] (14,8.5) -- (14,8);
\draw [->,ultra thick,red,dashed] (18,10) to [out=210,in=-10] (15.4,9.4);
\draw [ultra thick,red,dashed] (15.4,9.4) to [out=170,in=-30] (14,10);
\draw [->,ultra thick,green,dashed] (10,8) -- (11.4,9.4);
\draw [ultra thick,green,dashed] (11.4,9.4) -- (12,10);
\draw [->,ultra thick,green,dashed] (12,10) to [out=-65,in=135] (14,7.5);
\draw [ultra thick,green,dashed] (14,7.5) to [out=-45,in=145] (16,6);
\draw [->,ultra thick,green,dashed] (16,10) -- (17.4,8.6); 
\draw [ultra thick,green,dashed] (17.4,8.6) -- (18,8);
\draw [ultra thick,blue,dashed] (16,8) to [out=145,in=-50] (13.3,10);
\draw [->,ultra thick,blue,dashed] (13.3,6) to [out=130,in=-55] (12.7,6.7);
\draw [ultra thick,blue,dashed] (12.7,6.7) to [out=125,in=-65] (12,8);
\draw [ultra thick,blue,dashed] (16,8) to [out=65,in=-90] (16.8,10);
\draw [->,ultra thick,blue,dashed] (16.8,6) to [out=90,in=-80] (16.7,6.7);
\draw [ultra thick,blue,dashed] (16.7,6.7) to [out=100,in=-65] (16,8);
\draw [fill=red,red] (10, 6) circle [radius=0.15];
\draw [fill=red,red] (10, 10) circle [radius=0.15];
\draw [fill=red,red] (14, 6) circle [radius=0.15];
\draw [fill=red,red] (14, 8) circle [radius=0.15];
\draw [fill=red,red] (14, 10) circle [radius=0.15];
\draw [fill=red,red] (18, 6) circle [radius=0.15];
\draw [fill=red,red] (18, 10) circle [radius=0.15];
\draw [fill=green,green] (10, 8) circle [radius=0.15];
\draw [fill=green,green] (12, 6) circle [radius=0.15];
\draw [fill=green,green] (12, 10) circle [radius=0.15];
\draw [fill=green,green] (16, 6) circle [radius=0.15];
\draw [fill=green,green] (16, 10) circle [radius=0.15];
\draw [fill=green,green] (18, 8) circle [radius=0.15];
\draw [fill=blue,blue] (12, 8) circle [radius=0.15];
\draw [fill=blue,blue] (16, 8) circle [radius=0.15];
    \end{tikzpicture}
    \caption{In the toric trinity shown, the top right state belongs to a cyclic component of $14$ states, the bottom left state belongs to an acyclic component of $5$ states, and the bottom right state is an isolated point.}
    \label{fig:component_types}
\end{figure}
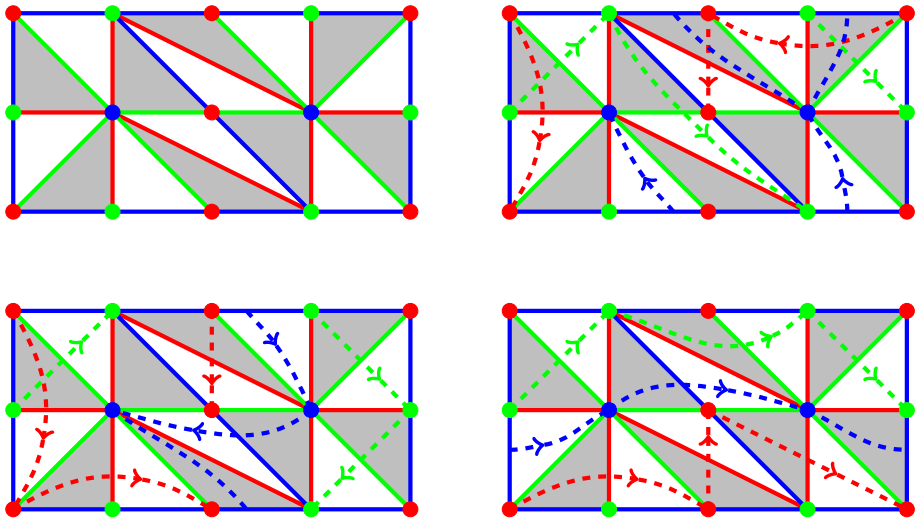

In Figure \ref{fig:component_types}, if we identify the top vertices and edges with the bottom vertices and edges, and identify the left vertices and edges with the right vertices and edges, we obtain a trinity on the torus. It has $28$ states and the state transition graph consists of $1$ cyclic component (of $14$ states) and $6$ acyclic components, $4$ of which are isolated points.
\end{ex}


All toric trinities have Tutte matchings. The reason is very similar to the spherical case, where an arborescence in one color determines the rest of the matching. Instead of arborescences, on the torus we consider collections of edges from one directed graph, $G_R^*$ say, so that every red vertex has exactly one edge of the collection pointing to it and the edges do not form any separating cycles. Let us call such a subgraph a \emph{wreath}. 

\begin{lemma}
Wreaths exist in any toric trinity.
\end{lemma}

\begin{proof}
Let $\mathcal T$ be a trinity on the torus $T^2$. Choose any non-zero homology class in $H_1(T^2;\Z)$ and represent it with a cycle $\Omega$ in the $1$-skeleton $G_R^*$. A priori this does not take into account the orientations of the edges. Note however that those edges of $G_R^*$ that are dual to the edges of $G_R$ incident to some fixed blue or green vertex $x$, form a directed cycle. 
This surrounds $x$ in the clockwise direction if $x$ is green and in the counter-clockwise direction if $x$ is blue; in particular, it is null-homologous. By suitable adding such cycles to $\Omega$ we may assume that it is also a directed cycle in $G_R^*$. 

If $\Omega$ has several components then at least one of those has to have a non-zero homology class, so by throwing the rest away we may assume that $\Omega$ is connected. Now at any double point of our cycle, it can be separated into two cycles, at least one of which still represents a non-zero class. Thus we may assume that $\Omega$ is an embedded, non-separating cycle.

Next we extend $\Omega$ into a wreath by adding on arborescences `hanging off' of it. More precisely, let us take a maximal collection $W$ of edges so that
\begin{enumerate}
\item it contains $\Omega$ as its only cycle, 
\item at most one of these edges points to any red vertex, and
\item any connected component not containing $\Omega$ is an isolated point.
\end{enumerate}
Then in the non-trivial connected component of $W$ all vertices have exactly one edge of $W$ pointing to them, for otherwise a path from the vertex to $\Omega$ would contain some vertex with two incoming edges.

Furthermore, as $G_R^*$ is connected and all of its vertices have the same in-degree and out-degree, the graph is Eulerian and in particular strongly connected. Hence if $W$ had an isolated point then we could consider a directed path from a vertex along $\Omega$ to it; the first edge along such a path that ends at an isolated point of $W$ could be used to make $W$ larger. From this contradiction it is clear that $W$ is a wreath.
\end{proof}

Unlike in the spherical case, a wreath for one color does not determine the matching uniquely, but it always has at least one extension.

\begin{thm}
Given a wreath, there are $2^k$ total Tutte matchings extending it, where $k$ is the number of cycles which the wreath contains.
\end{thm}

\begin{proof}
Let $W\subset G_R^*$ be a wreath. Its $k$ cycles are necessarily parallel on the torus and 
cut it up into $k$ annuli. The dual subgraph $W^*\subset G_R$ also has exactly $k$ cycles which form cores for the annuli. For blue and green vertices that do not lie along any of these cycles (that is, they lie along tree-like subgraphs attached to one of the cycles), there will be exactly one choice of incident white triangle (not already matched to a red vertex) just like in the spherical case. We see this by starting at leaves and working backward toward one of the core cycles. Finally the vertices along the cycles can be matched in exactly two ways for each cycle.
\end{proof}

\begin{rmk}
For two states, having vertices of one color matched to the same triangles in both puts the states in a relation. One may ask how this relation interacts with the other structures investigated in this paper. For example, it is possible for one connected component of the state transition graph to contain multiple states that are related. On the other hand, related states do not always lie in the same component and even the types of the components may be different.
\end{rmk}

Acyclic components associated with toric trinities are governed by Corollary \ref{cor:lattice} and hence they look exactly like state transition graphs of planar trinities. The novelty in toric cases is the presence of cyclic components. Let us investigate their structure more closely.

\begin{prop}
\label{prop:same_turns}
Suppose that a sequence of clockwise moves starts and ends with the same state of an irreducible trinity. Then in the sequence, each black triangle is changed the same number of times.
\end{prop}

\begin{proof}
Let us arbitrarily choose a vertex $x$ and observe that the arrow pointing to it makes $k$ full turns during our sequence of moves. This means that the black triangles incident to $x$ are each changed $k$ times (making $k$ full cycles in the cyclic order induced by the torus). The value of $k$ may a priori depend on $x$ but one quickly sees that the same value of $k$ applies to all neighbors of $x$ too, and then by induction and the connectedness of the $1$-skeleton of $\mathcal T$, to every vertex. Hence each black triangle is changed $k$ times.
\end{proof}

\begin{prop}
\label{recurrency}
In a cyclic component of the state transition graph, every state is recurrent.
\end{prop}

\begin{proof}
If the trinity decomposes as a connected sum, then by Theorem \ref{thm:graph_is_multiplicative} and Proposition \ref{prop:nocycle}, one of its components is cyclic if and only if the corresponding component of the state transition graph of the irreducible core is cyclic. We will be done if, starting from an arbitrary state, we find a cycle of moves in the irreducible core and do nothing in the planar parts of the connected sum. Hence from now on we may assume that the trinity is irreducible.

It suffices to show that neighbors (in the state transition graph) of recurrent states are themselves recurrent.
So let $s$ be a recurrent state, that is one to which it is possible to return by a non-empty sequence $\Omega$ of clockwise moves. Let also $s'$ be a state obtained from $s$ by a single clock move on the black triangle $\Delta$. Without loss of generality we may assume that the move is clockwise.

We claim that $s'$ is recurrent. This is obvious if $s'$ lies along $\Omega$. Otherwise 
by Proposition \ref{prop:same_turns} a clockwise move on $\Delta$ must appear somewhere else in the cycle $\Omega$. 
Let $a$ be the state which occurs immediately after a clockwise move is done on $\Delta$ in $\Omega$. 
Now if we move from $s'$ to $s$ and then follow $\Omega$ to $a$, then we have a path from $s'$ to $a$ in which the first move is a counter-clockwise move on $\Delta$, and all moves after it are clockwise moves. Since the last move is a clockwise move on $\Delta$, we can apply Lemma \ref{lem:cancel}, and we now have a path of only clockwise moves from $s'$ to $a$. By adding to this the rest of $\Omega$, from $a$ to $s$, and then finally a clockwise move on $\Delta$, we see that $s'$ is indeed recurrent.
\end{proof}

\begin{cor}
\label{reachability}
Starting from any state in a cyclic component of the state transition graph of an irreducible trinity, it is possible to reach any other state in the component by doing only clockwise moves.
\end{cor}

\begin{proof} 
Let $s$ and $s'$ be states that are connected by a path $\Pi$. As the component is cyclic, by Proposition \ref{recurrency} there exists a cycle $\Omega$ of clockwise moves starting and ending at $s'$. For every counter-clockwise move along $\Pi$, concatenate a copy of $\Omega$ to $\Pi$. Now by Proposition \ref{prop:same_turns}, the new path from $s$ to $s'$ contains a clockwise version of the same move for every counter-clockwise move. These can be canceled in pairs by Lemma \ref{lem:cancel}.
\end{proof}

Recurrent states must have at least one clockwise and counter-clockwise move (in the irreducible case, empty black triangle) available, in other words Proposition \ref{recurrency} implies that in cyclic components there are no fully clocked or unclocked states.

\begin{cor}
Cyclic components of the state transition graph of a toric trinity do not contain any local maxima or local minima.
\end{cor}


Proposition \ref{prop:same_turns} says that all cycle lengths (in the state transition graph of the irreducible trinity $\mathcal T$) are multiples of $n(\mathcal T)$, but it does not say how long the shortest cycles are. It turns out that the answer is $n$, even if we fix the starting point. That is, in irreducible trinities, the following stronger version of Proposition \ref{recurrency} is true.

\begin{thm}
\label{thm:one_turn}
In a cyclic component of the state transition graph of the irreducible toric trinity $\mathcal T$, for any state it is possible to find a linear order of the black triangles of $\mathcal T$ so that, starting from the given state, we may operate on the triangles in that order (and then the state automatically recurs at the end).
\end{thm}

\begin{proof}
From Proposition \ref{prop:same_turns} and its proof it is clear that our goal is to find a sequence of moves that gives one full turn, around its head, to each arrow of the matching 
in the counter-clockwise direction.

In a sense, the desired sequence constructs itself: as long as we have not yet operated on all black triangles, there will always be one that is currently empty (and clockwise) and is yet to be changed. To show this, let $s$ be a state in a cyclic component and suppose for contradiction that we have found a sequence $\Pi$ of clockwise moves about distinct black triangles so that the starting state $s$ has not yet recurred but the last state $s'$ is such that all of its clockwise empty black triangles have already occurred in $\Pi$. To simplify language, let us from now on consider $\Pi$ as a set of black triangles.

Let $x_0$ be any vertex so that the arrow pointing to it has not yet completed its turn. The next move of the arrow would be affected by changing the black triangle $\Delta_0$ to the right of the arrowhead. So we see that $\Delta_0$ has not been changed yet and by our assumption, it is currently (in $s'$) not empty.

Let $x_1$ be the next vertex of $\Delta_0$ in the counter-clockwise direction after $x_0$. See Figure \ref{fig:sequence}. Since $\Delta_0$ has not yet been changed, the arrow pointing to $x_1$ has not yet completed its turn. Hence we may iterate the argument of the previous paragraph and find the non-empty black triangle $\Delta_1$ incident to $x_1$. Since $\Delta_1$ is the only black triangle about which a clockwise move would move the arrow pointing to $x_1$ out of its current position, we see that the arrow can not became incident to $\Delta_0$, and hence $\Delta_0$ cannot become empty, before $\Delta_1$ becomes empty and is changed.

Continuing this way, we generate the sequence $x_0,\Delta_0,x_1,\Delta_1,x_2,\Delta_2,\ldots$ of vertices and non-empty (in the state $s'$) black triangles so that consecutive elements are incident and none of the $\Delta_i$ is in $\Pi$. Furthermore, in any sequence of clockwise moves starting from $s'$, the triangle $\Delta_i$ can not become empty before $\Delta_{i+1}$ becomes empty. Note also that the colors of the $x_i$ cycle through red, green, and blue with complete regularity.

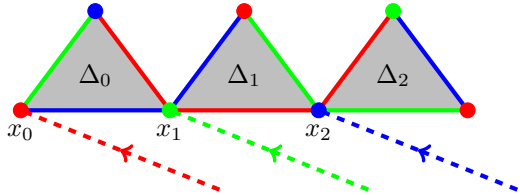
\begin{figure} [h]
\begin{tikzpicture}[scale=.66] 
\path [fill=lightgray] (0,0) -- (1.5,2) -- (3,0);
\path [fill=lightgray] (3,0) -- (4.5,2) -- (6,0);
\path [fill=lightgray] (6,0) -- (7.5,2) -- (9,0);
\draw [ultra thick, red] (1.5,2) -- (3,0) -- (6,0) -- (7.5,2);
\draw [ultra thick, green] (0,0) -- (1.5,2);
\draw [ultra thick, green] (4.5,2) -- (6,0) -- (9,0);
\draw [ultra thick, blue] (0,0) -- (3,0) -- (4.5,2);
\draw [ultra thick, blue] (7.5,2) -- (9,0);
\draw [->,ultra thick,red,dashed] (4,-1.6) to (2,-.8);
\draw [ultra thick,red,dashed] (2,-.8) to  (0,0);
\draw [->,ultra thick,green,dashed] (7,-1.6) to (5,-.8);
\draw [ultra thick,green,dashed] (5,-.8) to  (3,0);
\draw [->,ultra thick,blue,dashed] (10,-1.6) to (8,-.8);
\draw [ultra thick,blue,dashed] (8,-.8) to  (6,0);
\draw [fill=red,red] (0, 0) circle [radius=0.15];
\draw [fill=red,red] (4.5,2) circle [radius=0.15];
\draw [fill=red,red] (9,0) circle [radius=0.15];
\draw [fill=green,green] (3, 0) circle [radius=0.15];
\draw [fill=green,green] (7.5,2) circle [radius=0.15];
\draw [fill=blue,blue] (6, 0) circle [radius=0.15];
\draw [fill=blue,blue] (1.5,2) circle [radius=0.15];
\node [below] at (0,-.1) {$x_0$};
\node [below] at (3,-.1) {$x_1$};
\node [below] at (6,-.1) {$x_2$};
\node at (1.5,.7) {$\Delta_0$};
\node at (4.5,.7) {$\Delta_1$};
\node at (7.5,.7) {$\Delta_2$};
    \end{tikzpicture}
    \caption{Generating a sequence of `stuck triangles.'}
    \label{fig:sequence}
\end{figure}

Since $\mathcal T$ is finite, our sequence eventually becomes periodic. Without loss of generality we may assume that $x_{3k}=x_0$ for some natural number $k$ (and that $x_0\ne x_i$ for $0<i<3k$). This implies $\Delta_{3k}=\Delta_0$, too. In the case when $k=1$ and $\Delta_0=\Delta_1=\Delta_2$, we have just found the contradiction that $\Delta_0$ is a clockwise empty black triangle. Otherwise the sequence $\Delta_0,\ldots,\Delta_{3k-1}$ contains at least two triangles and we see that none of these triangles can be changed in any sequence of clockwise moves starting from $s'$.
By Proposition \ref{prop:same_turns} this implies that $s'$ is not a recurrent state, and that contradicts Proposition \ref{recurrency}. Therefore, in a cyclic component, the next clockwise move can always be found until all arrows complete their turns.
\end{proof}

There are many more elementary questions that one may ask about cyclic and acyclic components in toric trinities. Some of them can be answered easily: For example,
it is not hard to find trinities with multiple cyclic components in their state transition graphs. It is also not difficult to avoid isolated points --- in a sense, a `typical' trinity will not have them. Other problems however are still open. For instance, given a state, we are not aware of a quick test to decide the type of its component. Let us conclude with another natural question, namely: do all toric trinities have to have both types of components? If not, is one of the two types unavoidable or are there trinities with only cyclic components \emph{and} trinities with only acyclic ones?

\end{document}